\documentclass[12pt,reqno]{amsart}

\DeclareMathOperator{\ann}{ann}%
\DeclareMathOperator{\proj}{proj}%

\usepackage{amssymb}
\usepackage{hyperref}
\usepackage[all]{xypic} 
\usepackage{color}
\usepackage{tikz}
\usetikzlibrary{calc}

\newcommand{\F}{\mathbb{F}}

\newcommand{\Fp}{\F_p}

\newcommand{\Gal}{\text{\rm Gal}}
\newcommand{\Ic}{\mathcal{I}}

\newcommand{\N}{\mathbb{N}}

\newcommand{\Z}{\mathbb{Z}}
\newcommand{\comment}[1]{}

\makeindex

\begin{document}


\thanks{The first author is partially supported by the Natural Sciences and Engineering Research Council of Canada grant R0370A01.  He also gratefully acknowledges the Faculty of Science Distinguished Research Professorship, Western Science, in years 2004/2005 and 2020/2021. The second author is partially supported by 2017--2019 Wellesley College Faculty Awards. The third author was supported  in part by National Security Agency grant MDA904-02-1-0061.}

\title[On the indecomposability of a remarkable new family of modules]{On the indecomposability of a remarkable new family of modules appearing in Galois theory}

\author[J\'{a}n Min\'{a}\v{c}]{J\'{a}n Min\'{a}\v{c}}
\address{Department of Mathematics, Western University, London, Ontario, Canada N6A 5B7}
\email{minac@uwo.ca}

\author[Andrew Schultz]{Andrew Schultz}
\address{Department of Mathematics, Wellesley College, 106 Central Street, Wellesley, MA \ 02481 \ USA}
\email{andrew.c.schultz@gmail.com}

\author[John Swallow]{John Swallow}
\address{Office of the President, Carthage College, 2001 Alford Park Drive, Kenosha, WI \ 53140 \ USA}
\email{jswallow@carthage.edu}

\begin{abstract}
A powerful new perspective in the analysis of absolute Galois groups has recently emerged from the study of Galois modules related to classical parameterizing spaces of certain Galois extensions.  The recurring trend in these decompositions is their stunning simplicity: almost all summands are free over some quotient ring.  The non-free summands which appear are exceptional not only because they are different in form, but because they play the key role in controlling arithmetic conditions that allow the remaining summands to be easily described.  In this way, these exceptional summands are the lynchpin for a bevy of new properties of absolute Galois groups that have been gleaned from these surprising decompositions.

In one such recent decomposition, a remarkable new exceptional summand was discovered which exhibited interesting properties that have not been seen before.  The exceptional summand is drawn from a particular finite family that has not yet been investigated.  The main goal of this paper is to introduce this family of modules and verify their indecomposability.  We believe this module will be of interest to people working in Galois theory, representation theory, combinatorics, and general algebra.  The analysis of these modules includes some interesting new tools, including analogs of $p$-adic expansions.  
\end{abstract}

\date{\today}

\maketitle

\newtheorem*{theorem*}{Theorem}
\newtheorem*{lemma*}{Lemma}
\newtheorem{theorem}{Theorem}
\newtheorem{proposition}{Proposition}[section]
\newtheorem{corollary}[proposition]{Corollary}
\newtheorem{lemma}[proposition]{Lemma}

\theoremstyle{definition}
\newtheorem*{definition*}{Definition}
\newtheorem*{remark*}{Remark}

\parskip=10pt plus 2pt minus 2pt

\section{Introduction}

\subsection{Motivation}

The inverse Galois problem asks: for a given group $G$ and field $K$, does there exist some extension $L/K$ so that $\Gal(L/K) \simeq G$?  This problem is extremely challenging and is one of the most important outstanding problems in mathematics.  The most satisfying solution to this problem would be to have an explicit parameterization of all $G$-extensions of $K$ (i.e., extensions $L/K$ with $\Gal(L/K) \simeq G$).  In certain cases, such a parameterization is already known.  For instance, if $p$ is prime, the elementary $p$-abelian extensions of $K$ are parameterized by $1$-dimensional $\mathbb{F}_p$-subspaces of an $\mathbb{F}_p$-vector space $J(K)$.  The exact structure of $J(K)$ depends on some field-theoretic properties of $K$. The most familiar situation is when $\text{char}(K) \neq p$ and $K$ contains a primitive $p$th root of unity, in which case Kummer theory tells us that $J(K) = K^\times/K^{\times p}$.  

For a given field $K$, there is an object which encodes the information of all Galois extensions of $K$ simultaneously: the absolute Galois group of $K$.  This is the Galois group associated to the separable closure of $K$, and it is typically denoted $G_K$. It is a profinite group. By Galois theory, the extensions of $K$ with Galois group $G$ are determined by continuous surjections $G_K \twoheadrightarrow G$.  Since absolute Galois groups give complete insight into the inverse Galois problem, it should not be surprising to hear that for a general field $K$, computing $G_K$ is not currently feasible.  (For example, the group $G_\mathbb{Q}$ is much sought after.) Because of their connection to the inverse Galois problem (as well as other arithmetic properties of a field $K$), understanding any information about absolute Galois groups is highly desirable.  For instance, we can ask for a kind of realizability question for absolute Galois groups: for a given profinite group $G$, is there some field $K$ so that $G_K \simeq G$?  Or perhaps more loosely: can we distinguish the collection of absolute Galois groups among the collection of profinite groups?

One method for doing this is to first focus a bit more narrowly on the class of $p$-groups, in which case there is a corresponding classifying Galois object: the maximal pro-$p$ quotient of $G_K$.  We will denote this by $G_K(p)$, and call it the absolute $p$-Galois group.  It is a pro-$p$ group.  As with absolute Galois groups, one is very interested in understanding the ways in which absolute $p$-Galois groups are ``special" within the collection of pro-$p$ groups.  Although computing specific groups $G_K(p)$ remains extremely challenging, there are a number of ways in which these absolute $p$-Galois groups distinguish themselves.  Readers who are interested in exploring some concrete manifestations of these distinguishing properties may consult \cite{Jensen1,Jensen2,Jensen3,JensenPrestel1,JensenPrestel2,JensenLedetYui,Kiming,Ledet,Lorenz}.

Recently,  decompositions of a number of Galois modules have provided new insight into structural properties of absolute Galois groups.  We will explore examples that illustrate this claim below, but for now we point out some trends that have emerged in these decompositions. Most conspicuously, these modules are far simpler than one would expect \emph{a priori}, with the vast majority of summands free over some appropriately chosen quotient ring.  This simpler-than-expected structure can then be translated into properties that distinguish absolute $p$-Galois groups from the larger class of pro-$p$ groups. Another extremely interesting feature of these modules is the appearance of certain ``exceptional" summands which are discovered \emph{en route} to understanding particular arithmetic conditions related to the module.  These summands are exceptional on a number of fronts: they are different in form from the other summands, as they are often non-free; they are far more scarce than other summand types, often with only a single exceptional summand appearing in the entire module; and their connection to critical arithmetic conditions makes them pivotal in controlling behaviors in the rest of the module.  In this way, these exceptional summands are the lynchpin for a bevy of new properties of absolute Galois groups that have been gleaned from these surprising decompositions.

To illustrate the kinds of modules described in the previous paragraph, we start with the simplest case.  Note that the group $G_K(p)$ has a maximal elementary $p$-abelian quotient.  If we assume that $K$ contains a primitive $p$th root of unity, then Kummer theory tells us that the maximal elementary $p$-abelian quotient corresponds to our old friend $J(K) = K^\times/K^{\times p}$.  If $K$ is itself a Galois extension of a field $F$, then the natural action of $\Gal(K/F)$ descends to an action on $J(K)$, and so $J(K)$ becomes an $\mathbb{F}_p[\Gal(K/F)]$-module.  The $\mathbb{F}_p[\Gal(K/F)]$-module structure of $J(K)$, therefore, becomes an avenue for exploring the structure of $G_K(p)$.

The Galois module structure of $J(K)$ has been examined quite extensively in the case where $\Gal(K/F)$ is a cyclic $p$-group (see \cite{BS,Bo,F,MSS,MS,Schultz}).  To illustrate this, consider the case where $K$ contains a primitive $p$th root of unity, and let $n \in \mathbb{N}$ be given so that $\Gal(K/F) \simeq \mathbb{Z}/p^n\mathbb{Z}$.  We let $K_i$ denote the intermediate extension in $K/F$ for which $[K_i:F]=p^i$, and denote the corresponding Galois group $\Gal(K_i/F)$ by $G_i$.  In \cite{MSS} it is shown that with at most one exception, each summand of $J(K)$ is isomorphic to $\mathbb{F}_p[G_i]$ for some $0 \leq i \leq n$.  The one exception to this rule --- the ``exceptional summand" --- is cyclic of dimension $p^{i(K/F)}+1$ for some $i(K/F) \in \{-\infty,0,1,\cdots,n-1\}$.  Though the exceptional summand is an outlier in terms of its module structure, it plays a critical role in controlling certain arithmetic conditions that allow the remaining summands to be ``free".  (In the cases where $K$ does not contain a primitive $p$th root of unity, the decomposition of $J(K)$ is even simpler; see \cite{BS,Schultz}.)  

In addition to $J(K)$, a number of related modules have also been examined in the case where $\Gal(K/F)$ is a cyclic $p$-group.   For example, when $K$ contains a primitive $p$th root of unity, it is well known that $K^\times/K^{\times p} \simeq H^1(G_K(p),\mathbb{F}_p)$.  As such, it is natural to ask for the structure of higher cohomology groups as Galois modules as well.  These have been explored in \cite{LMSS,LMS}, where the connections between Galois cohomology and quotients of Milnor $K$-groups are used in powerful ways.  Further investigations into Galois modules related to Milnor $K$-groups in the characteristic $p$ case are also considered in \cite{BLMS}.  Though the decomposition of these other Galois modules might vary from the structural properties of $J(K)$ in certain details, they conform to the same spirit: most summands are free over quotients of $\Gal(K/F)$, with relatively few ``exceptional" summands that help keep track of arithmetic information that is critical to the decomposition.

The Galois modules described above have been connected to explicit structural properties of some Galois groups and absolute $p$-Galois groups, as well as $p$-Sylow subgroups of absolute Galois groups.  These include module-theoretic parameterizations for certain families of $p$-groups \cite{MSS.auto,MS2,Schultz,Waterhouse}, automatic realization results \cite{CMSHp3,MSS.auto,MS2,Schultz}, realization multiplicity results \cite{BS,CMSHp3,Schultz}, and even explicit restrictions on the generators and relations of such a group \cite{BLMS.prop.groups}.  In the very interesting paper \cite{BarySorokerJardenNeftin}, the authors use various Galois modules to make substantial progress towards determining the explicit structure of $p$-Sylow subgroups of absolute Galois groups of $\mathbb{Q}$ for odd primes $p$.  An important further ingredient and motivation for this work is the paper \cite{Labute}.

Given the success and applicability of previous computations, determining the structure of additional modules seems a promising way to gain deeper insights into the structure of absolute $p$-Galois groups.  One path for new exploration comes from considering more complex group actions, say when $\Gal(K/F)$ is a non-cyclic $p$-group.  Another route would be to keep $\Gal(K/F)$ a cyclic $p$-group, but allow more nuance in the modules themselves; a natural first candidate are the power classes $K^\times/K^{\times p^m}$, where $m$ is an integer larger than $1$.  In both cases, the daunting roadblock is that the modular representation theory becomes foreboding: whereas there are only finitely many isomorphism classes of indecomposable $\mathbb{F}_p[\mathbb{Z}/p^n\mathbb{Z}]$-modules, analogous classifications are typically not available in these more complex settings.  More precisely, if $G$ is a non-cyclic $p$-group other than the Klein $4$-group, then the representation type of $\mathbb{F}_p[G]$-modules is wild, making the classification of all indecomposable types essentially hopeless.  Things are somewhat better for $\mathbb{Z}/p^m\mathbb{Z}[\mathbb{Z}/p^n\mathbb{Z}]$-modules when $m>1$.  In \cite{Szekeres}, Szekeres classified the indecomposables when $m>1$ and $n=1$ (though his work is under the guise of classifying metabelian groups).  Additional results in the case where $m,n \geq 2$ have been determined in \cite{DrobotenkoEtAl,HannulaThesis,Hannula68,Thevenaz81}, but only in the case where the $\mathbb{Z}/p^m\mathbb{Z}$-action is free.  A general classification of indecomposable types for $m,n \geq 2$ is not known.  Readers interested in reviewing some of the work on the representation theory of group rings can consult \cite{Benson,Carlson} as well as \cite{Bresar,ErdmannWildon,Nakano}.

Despite the modular-theoretic complexities, Galois module decompositions in these more complex settings have recently been determined.  In collaboration with Chemotti, the authors of this paper determined the structure of $K^\times/K^{\times 2}$ when $\Gal(K/F)$ is the Klein $4$-group in \cite{CMSS}.  Two of the authors of this paper, in collaboration with Heller, Nguyen, and T\^an, determined the structure of $J(K)$ when $G$ is \emph{any} finite $p$-group under the assumption that $G_K(p)$ is a finitely-generated, free pro-$p$ group, at least in the case where $\xi_p \in F$ or $\text{char}(F) = p$ (see \cite{HMS}).  The spirit of the previous results holds up even in these more forbidding settings: most summands are free over quotients of $\Gal(K/F)$, with  exceptional summands helping to manage certain arithmetic conditions that ensure other summands are free.  In both of these papers, the indecomposable summands which appear have an isomorphism type that is already well known.

In \cite{MSS2b}, the authors of this paper compute the structure of $K^\times/K^{\times p^m}$ over $\mathbb{Z}/p^m\mathbb{Z}[\Gal(K/F)]$, again in the case where $\Gal(K/F)$ is a cyclic $p$-group. The themes we have seen before continue to persist in this decomposition, with mostly free summands and a single exceptional summand.  Also as before, this exceptional summand helps control certain arithmetic conditions that are key to the module decomposition.  Indeed, the exceptional summand is uncovered in an attempt to find an integer $d \equiv 1 \pmod{p}$ and elements $\alpha,\delta_0,\dots,\delta_m \in K^\times$ which are appropriately ``minimal" and that simultaneously satisfy
\begin{equation}\label{eq:norm.minimizing.equation}
\begin{split}
\alpha^{\sigma} &= \alpha^{d} \delta_0 \delta_1^p \cdots\delta_{m}^{p^m}\\
\xi_p &= N_{K/F}(\alpha)^{\frac{d-1}{p}}\cdot N_{K_{n-1}/F}(\delta_0) \cdot \left(\prod_{i=1}^{m} N_{K/F}(\delta_i)^{p^{i-1}} \right).
\end{split}
\end{equation}
Minimal, in this case, is defined in terms of the intermediate fields which contain the elements $\delta_i$, as well as the integer $d$.  Specifically, let $a_i=-\infty$ if $\delta_i=1$, and otherwise let $a_i \in \{0,1,\cdots,n\}$ be chosen minimally so that $\delta_i$ is contained in the intermediate field of degree $p^{a_i}$ over $F$ within $K/F$.  Then let $\mathbf{a}$ be the vector $(a_0,a_1,\cdots,a_{m-1})$.  In finding a minimal solution to equation (\ref{eq:norm.minimizing.equation}), one first attempts to minimize the vector $\mathbf{a}$ lexicographically, and then one attempts to maximize the $p$-adic valuation of $d-1$.  So, for instance, the vector $\mathbf{a}_1 = (2,5,10,11)$ is ``smaller" than the vector $\mathbf{a}_2 = (2,6,7,9)$, and for a fixed choice of $(a_1,a_1,\cdots,a_{m-1})$ one would prefer $d=1+p^6$ to $\tilde d = 1+p^3+p^7$. Having a minimal solution to these equations, it turns out, is  key to an inductive approach to understanding the structure of $K^\times/K^{\times p^m}$, since it is central to understanding any ``new" dependencies that arise as one moves from $K^\times/K^{\times p^{m-1}}$ to $K^\times/K^{\times p^m}$.  

In contrast to previous Galois module decompositions, the modules which appear as an exceptional summand in the Galois module structure of $K^\times/K^{\times p^m}$ have not been previously investigated.  More specifically, it is far from clear that they are indecomposable.  The purpose of this paper is to verify their indecomposability.  

In a larger sense, however, this paper is a testament to the interest and importance of this family of modules.  In the untamable wilderness of $\mathbb{Z}/p^m\mathbb{Z}[\mathbb{Z}/p^n\mathbb{Z}]$-modules, the underlying issue is that there are simply too many indecomposable types to ever understand.  A more tractable goal, however, is to identify those indecomposables that appear in significant ways across the mathematical spectrum.  This is precisely what the modules $X_{\mathbf{a},d,m}$ do, playing a starring role in the structure of the foundational Galois modules $K^\times/K^{\times p^m}$.  

\subsection{Statement of the main theorem}

Let $G$ denote an abstract cyclic group of order $p^n$,
$n\in \N$, with fixed generator $\sigma$. For $0\le i\le n$,
let $G_i$ denote  the quotient group of $G$ of order $p^i$.  For $m\in \N$ we
set $R_m:=\Z/p^m\Z$ and let $R_mG$ denote the group ring. For $i\in \N$ set $U_i=1+p^i\Z$ and additionally set $U_\infty=\{1\}$. We adopt several conventions, including $p^{-\infty}=0$ and that $\{0\}$ is a free module over any ring.  Moreover, if $M$ is any additive $R_mG$-module and $x \in M$ is given, we define $\sigma^{p^{-\infty}}x$ to be $0_M$.

For $\mathbf{a}\in \{-\infty,0,1,\dots,n\}^m$ and $d\in
\Z$, we define the (additive) $R_mG$-module $X_{\mathbf{a},d,m}$ according to the following generators and relations:
\begin{equation}\label{eq:definition.of.X}
\begin{split}
   \left\langle y, x_0, \dots, x_{m-1} :
    \ (\sigma-d)y=\sum_{i=0}^{m-1} p^ix_i~;
    \ \sigma^{p^{a_i}}x_i=x_i, \ 0\le i<m\right\rangle.
\end{split}
\end{equation}
In particular, note that for each $i$ we have $\langle x_i \rangle \simeq R_mG_{a_i}$.  This relation also gives us a convenient way to keep track of the elements in $X_{\mathbf{a},d,m}$ in a manner that is reminiscent of $p$-adic expansions.  Indeed, we have adapted some techniques from $p$-adic analysis into this representation-theoretic setting; some excellent resources on the $p$-adic numbers are available in \cite{Gouvea,Katok,Koblitz}. For example, since $\langle x_i \rangle \simeq R_mG_{a_i}$ it is not difficult to show that for any element $f(\sigma) \in R_mG_{a_i}$ there exist unique quantities $c_{k,\ell} \in \{0,1,\cdots,p-1\}$ so that $$f(\sigma) x_i = \sum_{k=0}^{p^{a_i}-1} \sum_{\ell=0}^{m-1} c_{k,\ell} p^\ell (\sigma-1)^{k} x_i.$$  The given relation (\ref{eq:definition.of.X}) makes such a decomposition also possible for elements in $\langle y \rangle$, this time expressing images of $y$ under the $R_mG$ action as ``$\mathbb{F}_p$-combinations" drawn from $\{y^\ell: 0 \leq \ell \leq m-1\} \cup \{p^\ell(\sigma-1)^k  \sum_{i=0}^{m-1} p^i x_i\}$.

In Figure \ref{fig:relations.depiction} we give a graphical depiction of this information in the special case $p=2$, $\mathbf{a} = (1,2,3,4)$, $d=1$, and $m=4$. In this depiction we have several 2-d arrays of boxes: one for each of $\langle x_0\rangle, \langle x_1\rangle, \langle x_2 \rangle$, and $\langle x_3\rangle$, as well as for the module $\langle y,x_0,x_1,x_2,x_3 \rangle$.  For each, progression from left to right along boxes indicates successive powers of $\sigma-1$, and progression from top to bottom indicates successive powers of $p$.  The highlighted box in each $\langle x_i\rangle$ indicates the contribution from each module to the value of $(\sigma-1)y$; the overall module is depicted as a ``layered" combination of the various modules.\footnote{Note that the choice of $\mathbf{a}$ in this graphical depiction violates one of the conditions of Theorem \ref{th:indecomp}, and is therefore not one of the modules we ultimately show is indecomposable.  Nevertheless, it gives a good indication of the ``shape" of these kinds of modules while still being small enough to reasonably fit on the page.} 
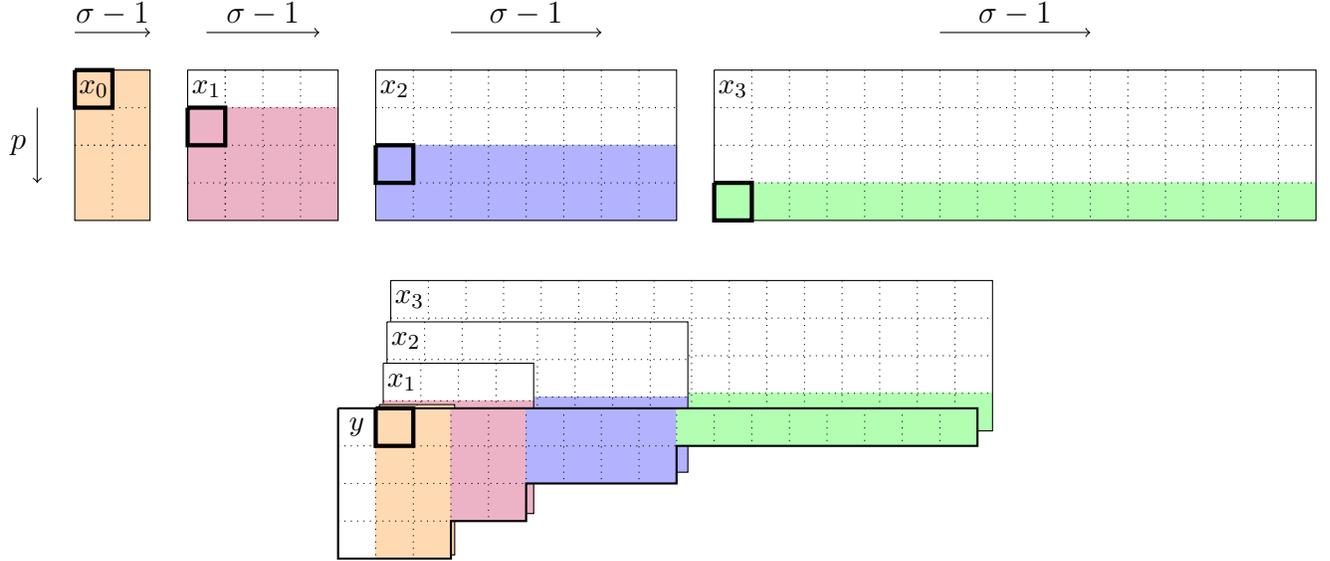
\begin{figure}[ht!]
\begin{tikzpicture}[scale=0.5]

\coordinate (A) at (14,0);
\filldraw[white] (A)--++(16,0)--++(0,-4)--++(-16,0)--++(0,4);
\filldraw[green!30!white] (A)++(0,-3)--++(16,0)--++(0,-1)--++(-16,0)--++(0,1);
\draw[-] (A)--++(16,0)--++(0,-4)--++(-16,0)--++(0,4);
\draw[-,ultra thick] (A)++(0,-3)--++(1,0)--++(0,-1)--++(-1,0)--++(0,1);
	\draw[-, dotted] (A)++(1,0)--++(0,-4);
	\draw[-, dotted] (A)++(2,0)--++(0,-4);
	\draw[-, dotted] (A)++(3,0)--++(0,-4);
	\draw[-, dotted] (A)++(4,0)--++(0,-4);
	\draw[-, dotted] (A)++(5,0)--++(0,-4);
	\draw[-, dotted] (A)++(6,0)--++(0,-4);
	\draw[-, dotted] (A)++(7,0)--++(0,-4);
	\draw[-, dotted] (A)++(8,0)--++(0,-4);
	\draw[-, dotted] (A)++(9,0)--++(0,-4);
	\draw[-, dotted] (A)++(10,0)--++(0,-4);
	\draw[-, dotted] (A)++(11,0)--++(0,-4);
	\draw[-, dotted] (A)++(12,0)--++(0,-4);
	\draw[-, dotted] (A)++(13,0)--++(0,-4);
	\draw[-, dotted] (A)++(14,0)--++(0,-4);
	\draw[-, dotted] (A)++(15,0)--++(0,-4);
	\draw[-,dotted] (A)++(0,-1)--++(16,0);
	\draw[-,dotted] (A)++(0,-2)--++(16,0);
	\draw[-,dotted] (A)++(0,-3)--++(16,0);
\node at ($(A)+(.5,-0.5)$) {\small $x_3$};
\node at ($(A)+(8,1.5)$) {$\sigma-1$};
\draw[->] ($(A)+(6,1)$) -- ($(A)+(10,1)$);

\coordinate (B) at (5,0);
\filldraw[white] (B)--++(8,0)--++(0,-4)--++(-8,0)--++(0,4);
\filldraw[blue!30!white] (B)++(0,-2)--++(8,0)--++(0,-2)--++(-8,0)--++(0,2);
\draw[-] (B)--++(8,0)--++(0,-4)--++(-8,0)--++(0,4);
\draw[-,ultra thick] (B)++(0,-2)--++(1,0)--++(0,-1)--++(-1,0)--++(0,1);	\draw[-,dotted] (B)++(1,0)--++(0,-4);
	\draw[-,dotted] (B)++(2,0)--++(0,-4);
	\draw[-,dotted] (B)++(3,0)--++(0,-4);
	\draw[-,dotted] (B)++(4,0)--++(0,-4);
	\draw[-,dotted] (B)++(5,0)--++(0,-4);
	\draw[-,dotted] (B)++(6,0)--++(0,-4);
	\draw[-,dotted] (B)++(7,0)--++(0,-4);
	\draw[-,dotted] (B)++(0,-1)--++(8,0);
	\draw[-,dotted] (B)++(0,-2)--++(8,0);
	\draw[-,dotted] (B)++(0,-3)--++(8,0);
\node at ($(B)+(.5,-0.5)$) {\small $x_2$};
\node at ($(B)+(4,1.5)$) {$\sigma-1$};
\draw[->] ($(B)+(2,1)$) -- ($(B)+(6,1)$);

\coordinate (C) at (0,0);
\filldraw[white] (C)--++(4,0)--++(0,-4)--++(-4,0)--++(0,4);
\filldraw[purple!30!white] (C)++(0,-1)--++(4,0)--++(0,-3)--++(-4,0)--++(0,3);
\draw[-] (C)--++(4,0)--++(0,-4)--++(-4,0)--++(0,4);
\draw[-,ultra thick] (C)++(0,-1)--++(1,0)--++(0,-1)--++(-1,0)--++(0,1);	\draw[-,dotted] (C)++(1,0)--++(0,-4);
	\draw[-,dotted] (C)++(2,0)--++(0,-4);
	\draw[-,dotted] (C)++(3,0)--++(0,-4);
	\draw[-,dotted] (C)++(0,-1)--++(4,0);
	\draw[-,dotted] (C)++(0,-2)--++(4,0);
	\draw[-,dotted] (C)++(0,-3)--++(4,0);
\node at ($(C)+(.5,-0.5)$) {\small $x_1$};
\node at ($(C)+(2,1.5)$) {$\sigma-1$};
\draw[->] ($(C)+(0.5,1)$) -- ($(C)+(3.5,1)$);

\coordinate (D) at (-3,0);
\filldraw[white] (D)--++(2,0)--++(0,-4)--++(-2,0)--++(0,4);
\filldraw[orange!30!white] (D)--++(2,0)--++(0,-4)--++(-2,0)--++(0,4);
\draw[-] (D)--++(2,0)--++(0,-4)--++(-2,0)--++(0,4);
\draw[-,ultra thick] (D)++(0,0)--++(1,0)--++(0,-1)--++(-1,0)--++(0,1);	\draw[-,dotted] (D)++(1,0)--++(0,-4);
	\draw[-,dotted] (D)++(0,-1)--++(2,0);
	\draw[-,dotted] (D)++(0,-2)--++(2,0);
	\draw[-,dotted] (D)++(0,-2)--++(2,0);
\node at ($(D)+(.5,-0.5)$) {\small $x_0$};
\node at ($(D)+(1,1.5)$) {$\sigma-1$};
\draw[->] ($(D)+(0,1)$) -- ($(D)+(2,1)$);
\node at ($(D)+(-1.5,-2)$) {$p$};
\draw[->] ($(D)+(-1,-1)$) -- ($(D)+(-1,-3)$);

\coordinate (A') at (5.4,-5.6);
\filldraw[white] (A')--++(16,0)--++(0,-4)--++(-16,0)--++(0,4);
\filldraw[green!30!white] (A')++(0,-3)--++(16,0)--++(0,-1)--++(-16,0)--++(0,1);
\draw[-] (A')--++(16,0)--++(0,-4)--++(-16,0)--++(0,4);
	\draw[-, dotted] (A')++(1,0)--++(0,-4);
	\draw[-, dotted] (A')++(2,0)--++(0,-4);
	\draw[-, dotted] (A')++(3,0)--++(0,-4);
	\draw[-, dotted] (A')++(4,0)--++(0,-4);
	\draw[-, dotted] (A')++(5,0)--++(0,-4);
	\draw[-, dotted] (A')++(6,0)--++(0,-4);
	\draw[-, dotted] (A')++(7,0)--++(0,-4);
	\draw[-, dotted] (A')++(8,0)--++(0,-4);
	\draw[-, dotted] (A')++(9,0)--++(0,-4);
	\draw[-, dotted] (A')++(10,0)--++(0,-4);
	\draw[-, dotted] (A')++(11,0)--++(0,-4);
	\draw[-, dotted] (A')++(12,0)--++(0,-4);
	\draw[-, dotted] (A')++(13,0)--++(0,-4);
	\draw[-, dotted] (A')++(14,0)--++(0,-4);
	\draw[-, dotted] (A')++(15,0)--++(0,-4);
	\draw[-,dotted] (A')++(0,-1)--++(16,0);
	\draw[-,dotted] (A')++(0,-2)--++(16,0);
	\draw[-,dotted] (A')++(0,-3)--++(16,0);
\node at ($(A')+(.5,-0.5)$) {\small $x_3$};

\coordinate (B') at (5.3,-6.7);
\filldraw[white] (B')--++(8,0)--++(0,-4)--++(-8,0)--++(0,4);
\filldraw[blue!30!white] (B')++(0,-2)--++(8,0)--++(0,-2)--++(-8,0)--++(0,2);
\draw[-] (B')--++(8,0)--++(0,-4)--++(-8,0)--++(0,4);
	\draw[-,dotted] (B')++(1,0)--++(0,-4);
	\draw[-,dotted] (B')++(2,0)--++(0,-4);
	\draw[-,dotted] (B')++(3,0)--++(0,-4);
	\draw[-,dotted] (B')++(4,0)--++(0,-4);
	\draw[-,dotted] (B')++(5,0)--++(0,-4);
	\draw[-,dotted] (B')++(6,0)--++(0,-4);
	\draw[-,dotted] (B')++(7,0)--++(0,-4);
	\draw[-,dotted] (B')++(0,-1)--++(8,0);
	\draw[-,dotted] (B')++(0,-2)--++(8,0);
	\draw[-,dotted] (B')++(0,-3)--++(8,0);
\node at ($(B')+(.5,-0.5)$) {\small $x_2$};

\coordinate (C') at (5.2,-7.8);
\filldraw[white] (C')--++(4,0)--++(0,-4)--++(-4,0)--++(0,4);
\filldraw[purple!30!white] (C')++(0,-1)--++(4,0)--++(0,-3)--++(-4,0)--++(0,3);
\draw[-] (C')--++(4,0)--++(0,-4)--++(-4,0)--++(0,4);
	\draw[-,dotted] (C')++(1,0)--++(0,-4);
	\draw[-,dotted] (C')++(2,0)--++(0,-4);
	\draw[-,dotted] (C')++(3,0)--++(0,-4);
	\draw[-,dotted] (C')++(0,-1)--++(4,0);
	\draw[-,dotted] (C')++(0,-2)--++(4,0);
	\draw[-,dotted] (C')++(0,-3)--++(4,0);
\node at ($(C')+(.5,-0.5)$) {\small $x_1$};

\coordinate (D) at (5.1,-8.9);
\filldraw[white] (D)--++(2,0)--++(0,-4)--++(-2,0)--++(0,4);
\filldraw[orange!30!white] (D)--++(2,0)--++(0,-4)--++(-2,0)--++(0,4);
\draw[-] (D)--++(2,0)--++(0,-4)--++(-2,0)--++(0,4);
	\draw[-,dotted] (D)++(1,0)--++(0,-4);
	\draw[-,dotted] (D)++(0,-1)--++(2,0);
	\draw[-,dotted] (D)++(0,-2)--++(2,0);
	\draw[-,dotted] (D)++(0,-2)--++(2,0);
\node at ($(D)+(.5,-0.5)$) {\small $x_0$};

\coordinate (E) at (4,-9);
\filldraw[white] (E)--++(17,0)--++(0,-1) --++(-8,0) --++(0,-1) --++(-4,0) --++(0,-1) --++(-2,0) --++(0,-1) --++(-3,0) --++(0,4);
\filldraw[orange!30!white] (E)++(1,0)--++(2,0)--++(0,-4)--++(-2,0)--++(0,4);
\filldraw[purple!30!white] (E)++(3,0)--++(2,0)--++(0,-3)--++(-2,0)--++(0,3);
\filldraw[blue!30!white] (E)++(5,0)--++(4,0)--++(0,-2)--++(-4,0)--++(0,2);
\filldraw[green!30!white] (E)++(9,0)--++(8,0)--++(0,-1)--++(-8,0)--++(0,1);
\draw[-,thick] (E)--++(17,0)--++(0,-1) --++(-8,0) --++(0,-1) --++(-4,0) --++(0,-1) --++(-2,0) --++(0,-1) --++(-3,0) --++(0,4);
\draw[-,ultra thick] (E)++(1,0)--++(1,0)--++(0,-1)--++(-1,0)--++(0,1);	\draw[-,dotted] (C)++(1,0)--++(0,-4);
	\draw[-,dotted] (E)++(1,0)--++(0,-4);
	\draw[-,dotted] (E)++(2,0)--++(0,-4);
	\draw[-,dotted] (E)++(3,0)--++(0,-3);
	\draw[-,dotted] (E)++(4,0)--++(0,-3);
	\draw[-,dotted] (E)++(5,0)--++(0,-2);
	\draw[-,dotted] (E)++(6,0)--++(0,-2);
	\draw[-,dotted] (E)++(7,0)--++(0,-2);
	\draw[-,dotted] (E)++(8,0)--++(0,-2);
	\draw[-,dotted] (E)++(9,0)--++(0,-1);
	\draw[-,dotted] (E)++(10,0)--++(0,-1);
	\draw[-,dotted] (E)++(11,0)--++(0,-1);
	\draw[-,dotted] (E)++(12,0)--++(0,-1);
	\draw[-,dotted] (E)++(13,0)--++(0,-1);
	\draw[-,dotted] (E)++(14,0)--++(0,-1);
	\draw[-,dotted] (E)++(15,0)--++(0,-1);
	\draw[-,dotted] (E)++(16,0)--++(0,-1);
	\draw[-,dotted] (E)++(0,-1)--++(9,0);
	\draw[-,dotted] (E)++(0,-2)--++(5,0);
	\draw[-,dotted] (E)++(0,-3)--++(3,0);
\node at ($(E)+(0.5,-0.5)$) {\small $y$};
\end{tikzpicture}
\caption{A graphical depiction of $X_{\mathbf{a},d,m}$ when $p=2$, $\mathbf{a}=(1,2,3,4)$, $d=1$, and $m=4$, as well as the cyclic module generated by each $x_i$ for $0 \leq i \leq 3$.}\label{fig:relations.depiction}
\end{figure}

Our main result in this paper is to show that under certain hypotheses on $\mathbf{a},d$ and $m$, the module $X_{\mathbf{a},d,m}$ is indecomposable.

\begin{theorem}\label{th:indecomp}\index{Theorem \ref{th:indecomp}}
    Suppose $m\in \N$, $\mathbf{a}=(a_0,\dots,a_{m-1}) \in \{-\infty,0,\dots,n\}^{m}$, and $d\in U_1$.  Assume further that $d^{p^{a_i}} \in U_{i+1}$ for all $0 \leq i <m$.  
    
    If $p$ is odd, and if $a_i+j<a_{i+j}$ for all $0 \leq i < m$ with $1 \leq j < m-i$ and $a_{i+j} \neq -\infty$, then $X_{\mathbf{a},d,m}$ is an indecomposable $R_mG$-module.
    
    Moreover, if $p=2$, suppose we have 
    \begin{itemize}
    \item if $n=1$, then $a_0=\infty$;
    \item if $0 \leq i < m$ and $1 \leq j < m-i$, then $a_i+j<a_{i+j}$, unless $d \not\in U_2$, $i=0$, and $a_i=0$, in which case $a_j \neq 0$ for all $1 \leq j < m$; and
    \item if $m \geq 2$, $d \in -U_v\setminus -U_{v+1}$ for some $v \geq 2$, and $a_0=0$, then $a_i > i-(v-1)$ for all $v \leq i < m$ and $a_i \neq -\infty$.
    \end{itemize}  Then $X_{\mathbf{a},d,m}$ is an indecomposable $R_mG$-module.
\end{theorem}

As mentioned above, the authors began investigating this module because it appears in the $R_mG$ decomposition of $J_m:=K^{\times}/K^{\times p^m}$ when $\Gal(K/F) \simeq G$ (see \cite{MSS2b}).  In particular, if $d \in \mathbb{Z}$ and $\mathbf{a} = (a_0,\cdots,a_{m-1})$ are chosen minimally so that there exists $\alpha \in K^\times$ and $\delta \in K_{a_i}^\times$ which solve Equation (\ref{eq:norm.minimizing.equation}), then the module generated by their classes in $K^\times/K^{\times p^m}$ --- i.e., $\langle [\alpha]_m,[\delta_0]_m,[\delta_1]_m,\cdots,[\delta_{m-1}]_m\rangle$ --- is isomorphic to $X_{\mathbf{a},d,m}$.  In fact, each of the conditions listed in Theorem \ref{th:indecomp} arose because it was shown to be a necessary condition for a minimal solution to Equation (\ref{eq:norm.minimizing.equation}): failure to adhere to any single condition meant that one could find a ``smaller" solution. Field-theoretic interpretations for the quantities $\mathbf{a}$ and $d$ are explored in \cite{MSS2c}.

In addition to showing that $X_{\mathbf{a},d,m}$ is indecomposable, we will show in Proposition \ref{pr:geta}
that --- under the hypotheses of Theorem \ref{th:indecomp} --- the quantity $\mathbf{a}$ is a module-theoretic invariant of $X_{\mathbf{a},d,m}$; in particular, this means that if we have $\tilde{\mathbf{a}}$ and $\tilde d$ satisfying the conditions of Theorem \ref{th:indecomp} and for which $\mathbf{a} \neq \tilde{\mathbf{a}}$, then $X_{\mathbf{a},d,m} \not\simeq X_{\tilde{\mathbf{a}},\tilde d,m}$. 

\subsection{Outline of paper}  In writing this paper, we have worked to make the presentation as self-contained and approachable --- and even as friendly --- as possible.  We feel, in fact, that this paper should be approachable for people with a  basic university-level knowledge of algebra.  

The next two sections help bring the reader up to speed on the basic vocabulary and background needed to move towards our ultimate goal.  In section \ref{se:normoperators} we investigate certain elements of $\Z[\sigma]$ that will play a crucial role in our results, computing their annihilators and developing a few technical results.  In section \ref{se:rmgmodules} we provide some basic results on $R_mG$-modules.  The proof of Theorem \ref{th:indecomp} commences in section \ref{se:indecomp.proof}.  We do this by first showing indecomposability in certain special cases. As we proceed into the general case, we study the behavior of elements of $X$ which satisfy particular conditions. Using these properties, we then proceed to show the indecomposability of $X$ inductively on $m$.  In section \ref{sec:uniqueness} we prove that $\mathbf{a}$ is a module-theoretic invariant of $X_{\mathbf{a},d,m}$.  In the final section of the paper, we discuss a few ways in which the conditions we give in Theorem \ref{th:indecomp} are necessary for indecomposability.

\section{Norm Operators}\label{se:normoperators}

In our study of the modules $X_{\mathbf{a},d,m}$, we will see that the norm operator (and some of its kin) play a key role in understanding their indecomposability.  In this section, we analyze these operators and their properties.

We continue to use the notation adopted in the previous section.  Additionally, for a set $S$ of elements in a ring $R$, we use the notation
$\langle S\rangle$ to denote the subring of $R$ generated by the elements of $S$.  We define several elements of $\Z[\sigma]$ and then consider their images under natural homomorphisms $\Z[\sigma]\to \Z G_s \to R_m G_s$ for $0\le s\le n$.

\begin{definition*}
For $0\le j\le i \leq n$ and $d \in U_1$, set
\begin{align*}
P(i,j) &:= \sum_{k=0}^{p^{i-j}-1} \sigma^{kp^j} &&& Q_d(i,j) &:= \sum_{k=0}^{p^{i-j}-1}
    (d^{p^j})^{p^{i-j}-1-k}(\sigma^{p^j})^k.
\end{align*}
\end{definition*}

Note that when $K/F$ is an extension so that $\Gal(K/F) = \langle \sigma \rangle \simeq \mathbb{Z}/p^n\mathbb{Z}$, then for an element $k \in K^\times$ we have $N_{K/F}(k) = \prod_{i=0}^{p^n-1} \sigma^i(k) = k^{\sum_{i=0}^{p^n-1} \sigma^i} = k^{P(n,0)}$.  In fact, if we let $K_i$ denote the intermediate extension of $K/F$ of degree $p^i$ over $F$, then by the same token one sees that for $k \in K_i^\times$ we have $N_{K_i/K_j}(k) = k^{P(i,j)}$.  It is for this reason that we call this family of polynomials \emph{norm operators}.

Clearly $Q_1(i,j)=P(i,j)$.  Note that since $d\in U_1$, we have
$Q_d(i,j)=P(i,j)\mod p\Z[\sigma]$.  Also observe that for $0\le k\le j\le i$, one has
\begin{align*}
    P(i,k) &= P(i,j)P(j,k) \\
    Q_d(i,k) &= Q_d(i,j)Q_d(j,k)
\end{align*}
as well as the following identities
\begin{align*}
    (\sigma^{p^j}-1)P(i,j) &= (\sigma^{p^i}-1) \\
    (\sigma^{p^j}-d^{p^j})Q_d(i,j) &= (\sigma^{p^i}-d^{p^{i}}).\\
\end{align*}
As a natural consequence of our convention $\sigma^{p^{-\infty}}=0$, one has $P(i,-\infty) =
Q_d(i,-\infty) = 1$ for all $i$.

Before demonstrating some more intricate results for the polynomials $P(i,j)$ and $Q_d(i,j)$, we will need to establish a few arithmetic results.

\begin{lemma}\label{le:upower}\index{Lemma \ref{le:upower}} 
    Suppose that $i\in \N$, and $d\in U_i$,
    and $j\ge 0$.  Then
    \begin{equation*}
        \begin{cases}
            d^{p^j}\in U_{i+j},
            &p>2 \text{\ or\ } i>1 \text{\ or\ } j=0\\
            d^{p^j}=1, &p=2, \ i=1, \ j>0, \ d=-1\\
            d^{p^j}\in U_{v+j},
            &p=2, \ i=1, \ j>0, \ d\in -U_v.
        \end{cases}
    \end{equation*}
	When additionally $d \not\in U_{i+1}$, then we have $d^{p^j} \not\in U_{i+j+1}$ in the first case; likewise if $d \not\in -U_{v+1}$ in the last case, then $d^{p^j} \not\in U_{v+j+1}$.
\end{lemma}
\begin{remark*}
     When $p=2$ and $i=1$, either $d=-1$ or $d\in -U_v\setminus
    -U_{v+1}$ for some $v\ge 2$.  Hence when
    $p=2$, $i=1$, and $j>0$, we always have $d^{p^j}\in U_{i+j+1}$.
\end{remark*}
\begin{proof}
    If $j=0$ the conclusion is identical to the hypothesis, so we
    may assume that $j>0$.  Assume first that $p>2$, and write
    $d=1+p^{i}x$ for $x\in \Z$.  Then we calculate
    \begin{equation*}
        d^p = (1+p^ix)^p = 1+p\cdot p^ix+\left(\begin{matrix}
            p\\2\end{matrix}\right)(p^ix)^2+\cdots+(p^ix)^p.
    \end{equation*}
    Observe that
    \begin{equation*}
        \left(\begin{matrix}p\\k\end{matrix}\right)(p^ix)^k
        = \begin{cases}
        &0\mod (p^{i+k+1}), \quad 2\le k<p, \ i\ge 2\\
        &0\mod (p^{i+k}), \quad \phantom{{}^{+1}}k=p, \ i\ge 2\\
        &0\mod (p^{i+k}), \quad \phantom{{}^{+1}}2\le k<p, \ i=1\\
        &0 \mod (p^{i+k-1}), \quad k=p, \ i=1.
        \end{cases}
    \end{equation*}
    Since $p>2$ we have $d^p = 1+p\cdot p^ix \bmod (p^{i+2})$ and we conclude
    that $d^p\in U_{i+1}$.  Moreover $d^p \not\in U_{i+2}$ if $d \not\in U_{i+1}$, since in this case $x \not\in p\Z$. Hence $d^{p^j}\in U_{i+j}$ for all $j\ge 1$ when $p>2$, with $d^{p^j} \not\in U_{i+j+1}$ if $d \not\in U_{i+1}$.

    Now suppose that $p=2$ and write $d=1+2^ix$ for some $x\in \Z$.  Then
    \begin{equation*}
        d^2 = (1+2^ix)^2 = 1+ 2\cdot 2^i x+ 2^{2i}x^2.
    \end{equation*}
    If $i\ge 2$, then $2i>i+1$ and so $d^2=1+2^{i+1}x \bmod (2^{2i})$, whence
    $d^{2}\in U_{i+1}$.  If $d\not\in U_{i+1}$, then we further have $d^2 \not\in U_{i+2}$ since $x \not\in 2\Z$.  This gives $d^{2^j}\in U_{i+j}$ for all $j\ge 1$ when $p=2$ and $i\ge 2$, and further that $d^{2^j} \not\in U_{i+j+1}$ if $d \not\in U_{i+1}$.

    We are left with the case $p=2$, $i=1$, and $j>0$.  If $d=-1$, then the
    conclusion is clear.  Otherwise let $d\in -U_v$ with
    $v\ge 2$.  We write $d=-1+2^vx$ for $x\in \Z$ and calculate
    \begin{equation*} (-1+2^vx)^2 = 1-2\cdot 2^vx+2^{2v}x^2. \end{equation*}
    Since $v\ge 2$, $2v>v+1$ and as before we have that $d^2\in
    U_{v+1}$; also as before, if $d \not\in -U_{v+1}$, then we have $d^2 \not\in -U_{v+2}$.  Hence $d^{2^j}\in U_{v+j}$
    for all $j\ge 1$ when $p=2$, $i=1$, $j>0$, and $d\neq -1$, with $d^{2^j} \not\in -U_{v+2}$ when $d \not\in -U_{v+1}$.
\end{proof}

One can be slightly more precise with the previous result.  We leave the following result to the reader; when we write $d = 1+p^i x$, it can be used to give a precise expression for the form of $d^{p^j}$ in terms of $x$.

\begin{lemma}\label{le:upowerpnot2}\index{Lemma \ref{le:upowerpnot2}}
    Suppose that $p>2$, $d\in U_i$, and $i\ge 1$.  Write $d=1+p^i x$.  Then there exist $f,g \in \Z$ so that
    \begin{equation*}
        d^{p^j} = 1+p^{i+j }x\left(1+fxp^i + gxp^{i+1}\right).
    \end{equation*}


    Now suppose that $p=2$ and $d\in U_i$ for $i\ge 1$.  Write $d=1+2^i x$.  Then there exists $c \in \Z$ so that
    \begin{equation*}
        d^{2^{j}} = 1+2^{i+j}x\left(1+2^{i-1}x + 2^ixc\right).
    \end{equation*}
    (Here we interpret $U_0=\Z$.)
\end{lemma}

We now turn to $P(i,j)$ and $Q_d(i,j)$. We start by determining some properties of these polynomials under various evaluation homomorphisms.  

\begin{definition*}
    For $d\in \Z$, we write $\phi_{d,n}$ for the homomorphism of additive groups
    $\Z G\to \Z$ induced by $\sigma^t\mapsto d^t$, $0\le t<p^n$. We will generally abbreviate this notation by writing $\phi_d$ in place of $\phi_{d,n}$, since the value of $n$ is implicit from the associated group $G$. The kernel is
    then the subgroup $\langle \sigma-d, \sigma^2-d^2, \dots, \sigma^{p^n-1} -
    d^{p^n-1}\rangle$, which is a subset of the ideal $\langle \sigma-d\rangle
    \subset \Z G$.

    Now if $m\in \N$ is such that $d^{p^n}=1 \bmod p^m$, we let
    $\phi_d^{(m)}:R_m G\to R_m$ be the induced map; this is a ring homomorphism
    with kernel precisely $\langle \sigma-d\rangle$.  Since
    $\phi_d^{(m)}(1)=1$, it is moreover an $R_m$-homomorphism.
\end{definition*}


\begin{lemma}\label{le:phi}\index{Lemma \ref{le:phi}}
    If $d \in U_1$ and $0\le j\leq i\le n$, then 
    \begin{equation*}
       \phi_d(P(i,j)) = 0 \mod p^{i-j}.
    \end{equation*}

    More precisely,
    \begin{enumerate}
        \item\label{it:phi2} If $p>2$, $d\in U_2$, or $j>0$, then
          for $0\le j<i$
        \begin{equation*}
            \phi_d(P(i,j)) = p^{i-j} \mod p^{i-j+1}.
        \end{equation*}
        \item\label{it:phidm1} If $p=2$, $d=-1$, $j=0$, and $i>0$, then
        \begin{equation*}
          \phi_d(P(i,0)) = 0.
        \end{equation*}
        \item\label{it:phi3} If $p=2$, $d\in -U_v\setminus -U_{v+1}$ for $v\ge 2$, $j=0$, and
          $i>0$, then
        \begin{equation*}
            \phi_d(P(i,0)) = 2^{i+v-1} \mod 2^{i+v}.
        \end{equation*}
    \end{enumerate}
\end{lemma}

\begin{proof}
    (\ref{it:phi2}).  If $d^{p^j}=1$, then
    $\phi_d(P(i,j))=p^{i-j}$ and we are done.

    Otherwise, observe that if $p=2$, then $d^{p^j}\in U_2$ when either $d\in U_2$ or $j>0$.  Let $d^{p^{j}}=1+p^vx$ for $v\in \N$ and $x\in
    \Z\setminus p\Z$.  Then by Lemma~\ref{le:upower}\index{Lemma \ref{le:upower}}
    we calculate that $d^{p^i}=1+p^{v+i-j}x+p^{v+i-j+1}y$ for some
    $y\in \Z$.  Furthermore,
    \begin{equation*}
        \phi_d(P(i,j)) = \frac{(d^{p^j})^{p^{i-j}}-1}{d^{p^j}-1}.
    \end{equation*}
    Hence
    \begin{equation*}
        p^{v}x \phi_d(P(i,j)) = p^{v+i-j}x + p^{v+i-j+1}y
    \end{equation*}
    and so
    \begin{equation*}
        p^vx (\phi_d(P(i,j))-p^{i-j}) = p^v p^{i-j+1}y.
    \end{equation*}
    Dividing by $p^v$ we have
    \begin{equation*}
        x (\phi_d(P(i,j))-p^{i-j}) = 0\mod p^{i-j+1}.
    \end{equation*}
    Since $x\not\in p\Z$ we deduce
    \begin{equation*}
        \phi_d(P(i,j)) = p^{i-j}\mod p^{i-j+1}
    \end{equation*}
    and we have (\ref{it:phi2}).

    (\ref{it:phidm1}).  If $d=-1$, then a simple calculation yields $\phi_d(P(i,0))=0$.

    (\ref{it:phi3}).  Assume that $d\in -U_v\setminus -U_{v+1}$.
    Write $d=-1+2^vx$ for $x\in \Z\setminus 2\Z$ and $v\in \N$.
    By Lemma~\ref{le:upower}\index{Lemma \ref{le:upower}}, $v\ge 2$ and $d^{2^i} =
    1+2^{v+i}x+2^{v+i+1}y$ for some $y\in \Z$.  We have
    \begin{equation*}
        \phi_d(P(i,0))=\frac{d^{2^i}-1}{d-1}.
    \end{equation*}
    Hence
    \begin{equation*}
        (d-1)\phi_d(P(i,0))=d^{2^i}-1
    \end{equation*}
    and so
    \begin{equation*}
        (-2+2^vx)\phi_d(P(i,0)) = 2^{v+i}x + 2^{v+i+1}y.
    \end{equation*}
    Dividing by $2$ and using the fact that $1-2^{v-1}x\not\in
    2\Z$ and so is a unit in $\Z/2^{v+i}\Z$, we obtain
    \begin{equation*}
        \phi_d(P(i,0)) = 2^{v+i-1}\mod 2^{v+i},
    \end{equation*}
    as desired.

    In all cases the initial statement follows.
\end{proof}

We now move on to study the images of $P(i,j)$ and $Q_d(i,j)$ within $R_mG$.  Of primary interest to us are their annihilators and their image under the natural $R_m$-homomorphisms $\chi_j: R_mG \to R_m G_j$.

\begin{lemma}\label{le:separate}\index{Lemma \ref{le:separate}}
    For $m\in \N$, $0\le k\le m$, and $0\le j<i\le n$,
    \begin{equation*}
        \langle P(i,j)\rangle \cap \langle p^k\rangle =
        \langle p^kP(i,j)\rangle
    \end{equation*}
    as ideals in $R_mG_i$.
\end{lemma}

\begin{proof}
    Clearly  $\langle p^kP(i,j)\rangle$ lies in the intersection.
    Consider then an arbitrary element $c\in \langle P(i,j)\rangle
    \cap \langle p^k\rangle$.
    Hence there exist $r, s\in R_mG_i$ such that
    \begin{equation*}
        c = p^k r = P(i,j)s.
    \end{equation*}
    Since $(\sigma^{p^j}-1)$ annihilates $P(i,j)$ in $R_mG_i$, we
    may assume without loss of generality that $s=\sum_{t=0}^{p^j-1}
    s_t\sigma^t$ for $s_t\in R_m$, $t=0, \dots, p^j-1$.  (More
    explicitly, we can replace $s$ by the result of taking an arbitrary
    lift of $s$ into $R_m[\sigma]$, dividing by $\sigma^{p^j}-1$,
    taking the remainder, and projecting back into $R_mG_i$.)

    Then
    \begin{equation*}
        P(i,j)s =\sum_{u=0}^{p^{i-j}-1} \sum_{t=0}^{p^j-1}
        s_t\sigma^{t+up^j}.
    \end{equation*}
    Since $R_mG_i$ is a free $R_m$-module with base $\{1, \sigma,
    \dots,\sigma^{p^i-1}\}$, we deduce that $P(i,j)s \not\in
    p^kR_mG_i$ unless $s\in p^kR_mG_i$.  Hence $c\in \langle
    p^kP(i,j)\rangle$.
\end{proof}

\begin{lemma}\label{le:kerbasic}\index{Lemma \ref{le:kerbasic}}
    Suppose $m\in \N$.
    For $0\le i\le n$ and $0\le k\le m$,
    \begin{equation*}
        \ann_{R_mG_i} p^k = \langle p^{m-k}\rangle.
    \end{equation*}

    For $0\le j<i\le n$ and $0\le k\le m$,
    \begin{equation*}
        \ann_{R_mG_i} p^k(\sigma^{p^j}-1)= \langle P(i,j),
        p^{m-k}\rangle.
    \end{equation*}
\end{lemma}

\begin{proof}
    Assume first that $0\le i\le n$ and $0\le k\le m$. Then since $R_mG_i$ is a free $R_m$-module, it follows that $\ann_{R_mG_i} p^k = \langle p^{m-k}\rangle$.

    We now consider the second statement.  Suppose additionally that $0\le j<i$.

    Suppose $k=0$.  Then, writing elements of $R_mG_i$ in terms of the $R_m$-base $\{1,\sigma,\dots,\sigma^{p^i-1}\}$, it follows easily that
    \begin{equation*}
        \ann_{R_mG_i} (\sigma^{p^j}-1) = \langle P(i,j) \rangle.
    \end{equation*}

    Now assume that $0\le k\le m$.  It is clear that $\langle P(i,j),p^{m-k}\rangle$ lies in $\ann_{R_mG_i} p^k (\sigma^{p^j}-1)$.  Assume then that $r \in \ann_{R_mG_i} p^k (\sigma^{p^j}-1)$. Then from the previous paragraph we have
    \begin{equation*}
        p^k r = P(i,j)s
    \end{equation*}
    for some $s\in R_mG_i$.  By Lemma~\ref{le:separate}\index{Lemma \ref{le:separate}}, $p^k r = p^k P(i,j)t$ for some $t\in R_mG_i$. Therefore $r-P(i,j)t$ lies in $\ann_{R_mG_i} p^k = \langle p^{m-k}\rangle$, and so
    \begin{equation*}
        r \in \langle P(i,j), p^{m-k} \rangle,
    \end{equation*}
    as desired.
\end{proof}

\begin{lemma}\label{le:phidb}\index{Lemma \ref{le:phidb}}
    Suppose $m\in \N$, $0\le i\le n$, and $d\in U_1$.  Then
    \begin{equation*}
        \ann_{R_mG_i} (\sigma-d) = \langle p^kQ_d(i,0)\rangle
    \end{equation*}
    where $k=\min\{v\ge 0 : p^v(d^{p^i}-1)=0 \mod p^m\}$.
\end{lemma}

\begin{proof}
    Let $k$ be defined as in the lemma.  Observe that in $R_mG_i$,
    \begin{equation*}
        (\sigma-d)p^kQ_d(i,0) = p^k(\sigma^{p^i}-d^{p^i})
         = p^k(1-d^{p^i}) = 0
    \end{equation*}
    so $\langle p^kQ_d(i,0)\rangle \subset \ann_{R_mG_i}
    (\sigma-d)$.

    Now suppose
    \begin{equation*}
        r := \sum_{l=0}^{p^i-1} c_l\sigma^{l} \in
        \ann_{R_mG_i} (\sigma-d), \quad c_l\in R_m.
    \end{equation*}
    Expanding out $(\sigma-d)r=0$ and considering coefficients of
    the $R_m$-base $\{\sigma^l\}_{l=0}^{p^i-1}$ of $R_mG_i$ we
    obtain equations over $R_m$
    \begin{align*}
        c_{l-1}-dc_l&=0 \mod p^m, \quad 1\le l\le p^i-1\\
        c_{p^i-1}-dc_0&=0 \mod p^m.
    \end{align*}
    We deduce
    \begin{equation*}
        c_{p^i-l}=d^{l}c_0 \mod p^m, \quad 1\le l\le p^i.
    \end{equation*}
    In particular, $c_0=d^{p^i}c_0\mod p^m$.  Hence
    $0=(d^{p^i}-1)c_0\mod p^m$.

    Now write $c_0=p^wx$ for $0\le w\le m$ and $x\in \Z\setminus
    p\Z$.  Then $x$ is a unit modulo $p^m$ and so
    $(d^{p^i}-1)p^w=0\mod p^m$.  Hence $w\ge k$ and $c_0\in \langle
    p^k\rangle \subset R_m$. Moreover,
    \begin{equation*}
        r = \sum_{l=0}^{p^i-1} c_l\sigma^l
        =c_0d\sum_{l=0}^{p^i-1} d^{p^i-l-1}\sigma^{l}
        =c_0dQ_d(i,0)\in \langle p^kQ_d(i,0)\rangle,
    \end{equation*}
    as desired.
\end{proof}

\begin{lemma}\label{le:qhomo}\index{Lemma \ref{le:qhomo}}
    Suppose $m\in \N$ and $0\le j< i\le n$.  Suppose that $d\in
    U_1$.  Then under the natural $R_m$-homomorphism $\chi_j:
    R_mG\to R_mG_j$,
    \begin{equation*}
        \chi_j(Q_d(i,0)) = 0 \mod p^{i-j}R_mG_j.
    \end{equation*}

    Moreover,
    \begin{enumerate}
      \item\label{it:qh1} If $p>2$, $d\in U_2$, or $j>0$, then
      \begin{equation*}
          \chi_j(Q_d(i,0)) = p^{i-j} Q_d(j,0) \mod p^{i-j+1}R_mG_j.
      \end{equation*}
      \item\label{it:qh2} If $p=2$, $d=-1$, and $j=0$, then
      \begin{equation*}
          \chi_j(Q_d(i,0)) = 0.
      \end{equation*}
      \item\label{it:qh3} If $p=2$, $d\in -U_v\setminus -U_{v+1}$ for $v \geq 2$, and
        $j=0$, then
      \begin{equation*}
          \chi_j(Q_d(i,0))=2^{i+v-1} \mod p^{i+v}R_m G_j.
      \end{equation*}
    \end{enumerate}
\end{lemma}

\begin{proof}
    We calculate
    \begin{align*}
        \chi_j(Q_d(i,0)) &= \chi_j(d^{p^{i}-1}+d^{p^{i}-2}\sigma +
        \cdots + \sigma^{p^{i}-1}) \\ &= (d^{p^{i}-p^j} +
        d^{p^{i}-2p^j} + \cdots+1) Q_d(j,0) \\ &=
        \phi_{(d^{p^j})}(P(i-j,0)) Q_d(j,0).
    \end{align*}
    The remainder follows from Lemma~\ref{le:phi}\index{Lemma \ref{le:phi}}.
\end{proof}

The balance of the section will be spent on two lemmas concerning annihilators of certain multiply-generated ideals in $R_mG_i$; they are somewhat technical, and the reader may choose to omit them on first pass.  
\begin{lemma}\label{le:kerint}\index{Lemma \ref{le:kerint}}
    Suppose $m\in \N$, $0\le i\le n$, and $t\in \N\cup\{0\}$.  Let
    $(b_j)_{j=0}^t$ be a decreasing sequence of elements of
    $\{0,\dots,m-1\}$ and $(c_j)_{j=0}^t$ be an increasing
    sequence of elements of $\{-\infty,0,\dots,i-1\}$.

   If $c_0=-\infty$, then
    \begin{multline*}
        \ann_{R_mG_i} \langle p^{b_0},
        \{p^{b_j}(\sigma^{p^{c_j}}-1)\}_{j=1}^t \rangle = \\
        \begin{cases}
        \langle p^{m-b_0}\rangle, & t=0 \\
        \langle p^{m-b_0}P(i,c_1),\cdots,p^{m-b_{t-1}}P(i,c_t),
        p^{m-b_t}\rangle, & t>0.
        \end{cases}
    \end{multline*}

    If $c_0\neq -\infty$, then
    \begin{multline*}
        \ann_{R_mG_i} \langle \{p^{b_j}(\sigma^{p^{c_j}}-1)\}_{j=0}^t
        \rangle =
        \\
        \begin{cases} \langle P(i,c_0),p^{m-b_0}\rangle, & t=0 \\
        \langle
        P(i,c_0),p^{m-b_0}P(i,c_1),\cdots,p^{m-b_{t-1}}P(i,c_t),
        p^{m-b_t}\rangle, & t>0. \\
        \end{cases}
    \end{multline*}
\end{lemma}

\begin{proof}
    We prove the result by induction on $t$.  The base case $t=0$ is Lemma~\ref{le:kerbasic}\index{Lemma \ref{le:kerbasic}}.  Assume then that for some $t\ge 1$, the result holds for all $l<t$.  We prove the result holds for $t$ as well.

    To simplify our discussion, for $0 \le v \le t$ we denote
    \begin{equation*}
    A_v := \begin{cases} 
    	\ann_{R_mG_i} \langle p^{b_0}, \{p^{b_j}(\sigma^{p^{c_j}}-1)\}_{j=1}^{v}\rangle, &c_0 = -\infty \\ 		 
		\ann_{R_mG_i}  \langle \{p^{b_j} (\sigma^{p^{c_j}}-1)\}_{j=0}^{v} \rangle, &c_0 \neq -\infty	 
    \end{cases}
    \end{equation*}
and 
    \begin{equation*}
        I_v :=
        \begin{cases}
            \langle p^{m-b_0}\rangle, & c_0=-\infty, v=0 \\
            \langle p^{m-b_0}P(i,c_1),\cdots,p^{m-b_v}\rangle, &
            c_0=-\infty, v>0 \\
            \langle P(i,c_0),p^{m-b_0}\rangle, & c_0\neq -\infty,
            v=0 \\
            \langle P(i,c_0), p^{m-b_0}P(i,c_1),
            \cdots,p^{m-b_v}\rangle, & c_0\neq -\infty, v>0
        \end{cases}
    \end{equation*}
    where the ideals are those given in the statement of the lemma.
    By induction $A_{t-1} = I_{t-1}$, and we claim that $A_t=I_t$.

    We start by showing $I_t \subset A_t$.  Since the $(b_j)$ are decreasing by hypothesis, we have $I_t \subset I_{t-1}$, and so $I_t \subset A_{t-1}$ as well.  Since $A_t =  A_{t-1} \cap \ann_{R_mG_i} \left(p^{b_t}(\sigma^{p^{c_t}}-1)\right)$, this means that we only need to verify that $I_t \subset \left(\ann_{R_mG_i} p^{b_t}(\sigma^{p^{c_t}}-1)\right)$; we show that each of the generators of $I_t$ is in this set.  This is obvious for the last generator of $I_t$, since $p^{m-b_t} \in \ann_{R_mG_i} p^{b_t}\subset \ann_{R_mG_i}p^{b_t}(\sigma^{p^{c_t}}-1)$.  The remaining generators of $I_t$ each contain a factor of the form $P(i,c_j)$ for $0 \leq j \leq t-1$.  Since the $(c_j)$ are increasing this implies $P(i,c_j) \in \ann_{R_mG_i}(\sigma^{p^{c_t}}-1) \subset \ann_{R_mG_i}p^{b_t}(\sigma^{p^{c_t}}-1)$, which in turn shows the remaining generators of $I_t$ also lie in the desired annihilator.  Hence $I_t \subset A_t$.

    Now we show that $A_t\subset I_t$.  Let $r\in A_t$. Now
    $A_t\subset A_{t-1}$ by definition, and $A_{t-1}=I_{t-1}$.
    So $r\in I_{t-1}$.

    Assume first that $t=1$.  Since $r \in I_{t-1}$ we have
    \begin{equation*}
        r = \begin{cases}
            p^{m-b_0}f_0, \ f_0\in R_mG_i, &c_0=-\infty\\
            P(i,c_0)f_0+p^{m-b_0}f_1, \ f_0, f_1\in R_mG_i, &c_0\neq
            -\infty.
        \end{cases}
    \end{equation*}
    In the case $c_0\neq -\infty$, we see that $P(i,c_0)f_0\in I_t$.
    Hence we may assume without loss of generality that if $t=1$ we
    have $r=p^{m-b_0}s$ for $s\in R_mG_i$.

    Now suppose that $t>1$. Again using the fact that $r \in I_{t-1}$, we have
    \begin{equation*}
        r = \begin{cases}
            p^{m-b_0}P(i,c_1)f_0+\cdots+p^{m-b_{t-1}}f_{t-1}, \ f_k\in
            R_mG_i, &c_0=-\infty\\
            P(i,c_0)f_0+\cdots+p^{m-b_{t-1}}f_t, \ f_k\in R_mG_i,
            &c_0\neq -\infty.
        \end{cases}
    \end{equation*}
    We see that all but the last summands in the expression for $r$
    are already contained in $I_t$. Hence we may assume without loss
    of generality that if $t>1$ we have $r=p^{m-b_{t-1}}s$ for $s\in
    R_mG_i$.

    In all cases, then, we have obtained $r=p^{m-b_{t-1}}s$ for $s\in
    R_mG_i$. Since $r\in A_t \subset \ann_{R_mG_i} p^{b_t}(\sigma^{p^{c_t}}-1)$,
    we have
    \begin{equation*}
        p^{b_t}r = p^{m-b_{t-1}+b_t}s\in \ann_{R_mG_i} (\sigma^{p^{c_t}}-1)
        = \langle P(i,c_t)\rangle,
    \end{equation*}
    by Lemma~\ref{le:kerbasic}\index{Lemma \ref{le:kerbasic}}.  
    Lemma~\ref{le:separate}\index{Lemma \ref{le:separate}} implies $p^{m-b_{t-1}+b_t}s \in \langle p^{m-b_{t-1}+b_t}P(i,c_t)\rangle$, and so another application of Lemma~\ref{le:kerbasic}\index{Lemma \ref{le:kerbasic}} gives $$r \in \langle p^{m-b_{t-1}}P(i,c_t),p^{m-b_t}\rangle$$ as desired.
\end{proof}

For the next lemma, denote by $\bar r\in R_{m-1}G_i$ the image of
$r\in R_mG_i$ under the natural $R_m$-homomorphism.  We show that if an element's image in $R_{m-1}G_i$ sits in a particular annihilator, then it can be expressed as a combination of certain norm operators.  Though quite technical, this result will be useful to us later.

\begin{lemma}\label{le:tech}\index{Lemma \ref{le:tech}}
    Suppose $m\ge 2$, $0\le i\le n$, and $t\in \N\cup \{0\}$.
    Let $\{b_j\}_{j=0}^t$ be a decreasing sequence of elements of
    $\{0,\dots,m-2\}$ and $\{c_j\}_{j=0}^t$ be an increasing
    sequence of elements of $\{-\infty,0,\dots,i-1\}$.  Assume further that $t>0$ when $c_0=-\infty$.

    Assume additionally that
    \begin{enumerate}
        \item\label{it:t2} $d^{p^i}\in U_m$;
        \item\label{it:t3} $b_t=0$;
        \item\label{it:t4} $\phi_d^{(m)}(P(i,c_j))\in p^{1+b_{j-1}}R_m$ for all $0 < j\le t$ with $c_j\neq -\infty$, and that $\phi_d^{(m)}(P(i,c_0)) =0$ if $c_0 \neq -\infty$.
    \end{enumerate}

    For $r\in R_{m}G_i$, suppose $\overline{(\sigma-d)r} \in R_{m-1}G_i$ satisfies
    \begin{equation*}
        \overline{(\sigma-d)r}\in
        \begin{cases}
            \ann_{R_{m-1}G_i} \langle p^{b_0}, \{
            p^{b_j}(\sigma^{p^{c_j}}-1)\}_{j=1}^t\rangle, &
            c_0=-\infty\\
            \ann_{R_{m-1}G_i} \langle
            \{p^{b_j}(\sigma^{p^{c_j}}-1)\}_{j=0}^t \rangle, &
            c_0\neq -\infty.
        \end{cases}
    \end{equation*}
    Then in $R_mG_i$ we have
    \begin{equation*}
        (\sigma-d)r\in \\
        \begin{cases}
            \langle p^{m-1-b_0}P(i,c_1),\cdots,& \\
            \quad p^{m-1-b_{t-1}}P(i,c_t),
            p^{m-1}(\sigma-d)\rangle, & c_0 = -\infty\\
            \langle P(i,c_0),p^{m-1-b_0}P(i,c_1),\cdots,& \\
            \quad p^{m-1-b_{t-1}}P(i,c_t), p^{m-1}(\sigma-d)\rangle,
            & c_0\neq -\infty.
        \end{cases}
    \end{equation*}
\end{lemma}

\begin{proof}
    Assume first that $c_0\neq -\infty$. By Lemma~\ref{le:kerint}\index{Lemma \ref{le:kerint}}
    for $R_{m-1}G_i$, using $b_t=0$ from item~(\ref{it:t3}), we have
    \begin{equation*}
        \overline{(\sigma-d)r} \in \langle
        \overline{P(i,c_0)}, \overline{p^{m-1-b_0}P(i,c_1)},
        \dots, \overline{p^{m-1-b_{t-1}}P(i,c_t)}\rangle
    \end{equation*}
    and so for some $f_0,\dots,f_{t+1}\in R_{m}G$ it follows that
   $$
        (\sigma-d)r = f_0P(i,c_0) + \left(\sum_{j=1}^t
        f_jp^{m-1-b_{j-1}}P(i,c_j)\right) + f_{t+1}p^{m-1}.
   $$

    By item~(\ref{it:t2}) we have the $R_m$-ring homomorphism
    $\phi_d^{(m)}$.  Then using item~(\ref{it:t4}) we have
    \begin{equation*}
        0 = p^{m-1}\phi_d^{(m)}(f_{t+1}).
    \end{equation*}
    By Lemma~\ref{le:kerbasic}\index{Lemma \ref{le:kerbasic}}, $f_{t+1}=(\sigma-d)g_{t+1}+ph_{t+1}$
    for $g_{t+1},h_{t+1}\in R_mG_i$ and we conclude that
    \begin{equation*}
        (\sigma-d)r \in \langle
        P(i,c_0),p^{m-1-b_0}P(i,c_1),\cdots,p^{m-1-b_{t-1}}P(i,c_t),
        p^{m-1}(\sigma-d)\rangle,
    \end{equation*}
    as desired.

    Now assume that $c_0= -\infty$. By Lemma~\ref{le:kerint}\index{Lemma \ref{le:kerint}},
    using $b_t=0$ from item~(\ref{it:t3}), we have
    \begin{equation*}
        \overline{(\sigma-d)r} \in \langle
        \overline{p^{m-1-b_0}P(i,c_1)}, \dots,
        \overline{p^{m-1-b_{t-1}}P(i,c_t)}\rangle
    \end{equation*}
    and so for some $f_1,\dots,f_{t+1}\in
        R_{m}G$ we have
    \begin{align*}
        (\sigma-d)r = \left(\sum_{j=1}^t f_jp^{m-1-b_{j-1}}
        P(i,c_j)\right) + f_{t+1}p^{m-1}.
    \end{align*}
    Just as before, by item~(\ref{it:t4}) we have
    \begin{equation*}
        0 = p^{m-1}\phi_d^{(m)}(f_{t+1})
    \end{equation*}
    and hence
    \begin{equation*}
        (\sigma-d)r \in \langle
        p^{m-1-b_0}P(i,c_1),\cdots,p^{m-1-b_{t-1}}P(i,c_t),
        p^{m-1}(\sigma-d)\rangle.
    \end{equation*}
\end{proof}

\section{$R_mG$-Modules}\label{se:rmgmodules}

In this section we consider some of the module-theoretic ingredients that will be used as we consider the indecomposability of $X_{\mathbf{a},d,m}$.

For all $m\in \N$ and $0\le i\le n$, the ring $R_mG_i$ is a local ring with unique maximal ideal $\Ic_{m,i} = \langle p,(\sigma-1)\rangle$. When the context is clear we abbreviate $\Ic_{m,i}$ by $\Ic$.  Since there is a unique maximal ideal, the Jacobson radical of $R_mG_i$ is simply $\mathcal{I}$ itself.  With this in mind, we have a special case of Nakayama's lemma that will prove extremely useful for the balance of the paper: if $M$ is a module of $R_mG_i$ such that $\mathcal{I}M = M$, then $M$ is trivial.  

We have that $U(R_mG_i)=R_mG_i\setminus \Ic$ and $U(R_m)\subset U(R_mG_i)$ for all $m\in \N$ and $0\le i\le n$. Recall that all cyclic modules over $R_mG_i$ are indecomposable, as follows.  If $M\neq \{0\}$ is cyclic, then $M$ is generated by a single element, and so $M/\Ic M\simeq \Fp$.  But if $M=A\oplus B$, then $M/\Ic M\simeq A/\Ic A\oplus B/\Ic B$.  Therefore without loss of generality $A/\Ic A=\{0\}$, and then by Nakayama's Lemma we have $A=\{0\}$.

For a $G$-module $M$ we write $M^G$ for the submodule of $M$
consisting of elements fixed by $G$.

If $H\le G$ are subgroups, it is clear that an $R_mG$-module $M$ is an $R_mH$-module.  On the other hand, if $H \le G$, then an $R_m(G/H)$-module $M$ is naturally an $R_mG$-module.  Let $H_i$ be the subgroup of $G$ of order $p^{n-i}$ so that $G_i=G/H_i$ is the quotient group of order $p^i$.  Then an $R_mG_i$-module is naturally an $R_mG$-module.

If $M=M^G$ we say that $M$ is a trivial $G$-module, and if an $R_mG$-module $M$ satisfies $M=M^G$, we say that $M$ is a trivial $R_mG$-module.  In particular, an $R_mG_0=R_m\{1\}$-module is a $R_mG$-module on which $G$ acts trivially.  Hence an $R_m\{1\}$-module is a trivial $R_mG$-module.

We define $l_M(u)$ of an element $u$ of an $R_mG$-module $M$ to be the dimension over $\Fp$ of the $\Fp G$-submodule $\langle \bar u\rangle$ of $M/pM$ generated by $\bar u:=u+pM$. When the context is clear we abbreviate $l_M(u)$ by $l(u)$.  We have
\begin{equation*}
    (\sigma-1)^{l(u)-1}\langle \bar u \rangle =
    \langle \bar u \rangle^G \neq \{0\} \text{\quad and\quad}
    (\sigma-1)^{l(u)}\langle \bar u \rangle = \{0\}.
\end{equation*}
Therefore $0\le l(u)\le p^n$.

We say that the \emph{length} of a cyclic $\Fp G$-module is the length of any generator of that module. Observe that as $\Fp G$-modules, the free $\Fp G_i$-modules on one generator are precisely the cyclic $\Fp G$-modules generated by an element of length $p^i$. Finally, recall that the indecomposable $\Fp G$-modules are precisely the cyclic $\Fp G$-modules.

We will use without mention the facts that in $\Fp G$,
\begin{equation*}
    P(i,j)=(\sigma^{p^j}-1)^{p^{i-j}-1}=(\sigma-1)^{p^{i}-p^j},
    \ 0\le j< i\le n.
\end{equation*}

\begin{lemma}\label{le:cycprop}\index{Lemma \ref{le:cycprop}}
    Let $m\ge 2$, $0\le i\le n$, and $M$ an $R_mG_i$-module.
    If $M/p^{m-1}M$ is a cyclic
    $R_{m-1}G_i$-module, then $M$ is a cyclic $R_mG_i$-module.
\end{lemma}

\begin{proof}
    Choose $u\in M$ such that $u$ generates $M/p^{m-1}M$.  Then for an arbitrary element $v\in M$, $v=fu+p^{m-1}x$ for $f\in R_mG$ and $x\in M$.  Moreover, $x=gu+p^{m-1}y$ for some $g\in R_mG$ and $y\in M$.  But then $v=(f+p^{m-1}g)u+ p^{2(m-1)}y=(f+p^{m-1}g)u$ since $2(m-1)\ge m$.  Hence $u$ generates $M$.
\end{proof}

For an $R_m G$-module $M$, we define the trivial $\Fp G$-module
\begin{equation*}
    M^\star := \ann_{M} \Ic = M^G \cap \ann_M p = \ann_M (\sigma-1) \cap \ann_M p.
\end{equation*}
For the balance of this section we will study this fixed submodule, and show that it can be used to great effect.  It is clear that $M\subset N$ implies $M^\star \subset N^\star$.  

\begin{lemma}\label{le:starzero}\index{Lemma \ref{le:starzero}}
    Let $M$ be a nonzero $R_mG$-module.  Then $M^\star\neq \{0\}$.
\end{lemma}

\begin{proof}
    Since every module contains a cyclic module, it is enough to show the result for a nonzero cyclic module $M=\langle u\rangle$. By Nakayama's Lemma we have $u\not\in pM$.  Now set
    \begin{equation*}
        j := \max \{ k\ge 0 \colon (\sigma-1)^k u \neq 0\},
    \end{equation*}
    where $(\sigma-1)^0=1$.  Hence $(\sigma-1)^j u \neq 0$. Now set
    \begin{equation*}
        l = \max \{ k\ge 0 \colon p^k(\sigma-1)^ju \neq 0\},
    \end{equation*}
    and let $w = p^l(\sigma-1)^j u$. Then by the maximality of $l$ we have $pw = 0$.  Moreover, by the maximality of $j$ we get $(\sigma-1)w = 0$ as well.  Therefore $0\neq w\in M^\star$.
\end{proof}

\begin{lemma}\label{le:excl}\index{Lemma \ref{le:excl}}
    Let $0\le i\le n$, and let $M_1$ and $M_2$ be $R_mG_i$-submodules contained in a common $R_mG_i$-module.  If $M_1^\star\cap M_2^\star = \{0\}$, then $M_1+M_2 = M_1\oplus M_2$.
\end{lemma}

\begin{proof}
    If $M_1\cap M_2 \neq \{0\}$, then by Lemma~\ref{le:starzero}\index{Lemma \ref{le:starzero}} we have $(M_1\cap M_2)^\star \neq 0$. Since $(M_1\cap M_2)^\star \subset M_1^\star\cap M_2^\star$, we have a contradiction.
\end{proof}

\begin{lemma}\label{le:starrmgi}\index{Lemma \ref{le:starrmgi}}
    For $m\in \N$ and $0\le i\le n$,
    \begin{equation*}
        (R_mG_i)^\star = \langle p^{m-1}(1+\sigma+\cdots+\sigma^{p^i-1})\rangle = \langle p^{m-1}(\sigma-1)^{p^i-1}\rangle \neq \{0\}.
    \end{equation*}
\end{lemma}

\begin{proof}
    From Lemma~\ref{le:kerbasic}\index{Lemma \ref{le:kerbasic}}, $\ann_{R_mG_i} (\sigma-1) = \langle P(i,0)\rangle$ and $\ann_{R_mG_i} p = \langle p^{m-1}\rangle$.  Then Lemma~\ref{le:separate}\index{Lemma \ref{le:separate}} gives us $(R_mG_i)^\star = \langle p^{m-1}(1 + \sigma + \cdots + \sigma^{p^i-1})\rangle$. Since $(\sigma-1)^{p^i-1} = 1 + \sigma + \cdots+\sigma^{p^{i}-1}$ modulo $p$, we have the second equality.  Now $R_mG_i$ is the free $R_m$-module on $\{1,\sigma,\cdots,\sigma^{p^i-1}\}$, which gives the inequality.
\end{proof}

\begin{lemma}\label{le:idealrmgi}\index{Lemma \ref{le:idealrmgi}}
    For $m\in \N$ and $0\le i\le n$, every nonzero ideal of $R_mG_i$
    contains
    \begin{equation*}
        p^{m-1}(\sigma-1)^{p^i-1}
        = p^{m-1}(1+\sigma+\cdots+\sigma^{p^{i}-1}).
    \end{equation*}
\end{lemma}

\begin{proof}
    Let $I$ be a nonzero ideal of $R_mG_i$.  Then $I$ is a
    submodule. Since $I$ is nonzero, Lemma~\ref{le:starzero}\index{Lemma \ref{le:starzero}} gives
    $I^\star \neq \{0\}$.
    On the other hand, Lemma~\ref{le:starrmgi}\index{Lemma \ref{le:starrmgi}} tells us
    \begin{equation*}
        I^\star \subset (R_mG_i)^\star = \langle
        p^{m-1}(\sigma-1)^{p^i-1}\rangle.
    \end{equation*}
    Therefore
    \begin{equation*}
        \{0\} \neq I^\star \subset \langle
        p^{m-1}(\sigma-1)^{p^i-1}\rangle.
    \end{equation*}
    Observing that the right-hand module is of $\Fp$-dimension 1, we have the result.
\end{proof}

As a final comment in this section, we point out that it will often be useful for us to know that two decompositions of an $R_mG$-module are equivalent to each other (up to the ordering of summands, of course).  This idea is made more precise by the Krull-Schmidt Theorem.

\begin{proposition}[Krull-Schmidt Theorem; {see \cite[Theorem~12.9]{AnFu}}]
Let $M$ be a nonzero module of finite length.  Then $M$ has a finite indecomposable decomposition $M = M_1 \oplus \cdots \oplus M_n$ such that for every indecomposable decomposition $$M = N_1 \oplus \cdots \oplus N_k,$$ we have $n=k$ and there is some permutation $\sigma \in S_n$ so that $M_i \simeq N_{\sigma(i)}$ for all $1 \leq i \leq n$.
\end{proposition}

\section{Indecomposability of $X_{\mathbf{a},d,m}$}\label{se:indecomp.proof}

We are now prepared to investigate the indecomposability of $X_{\mathbf{a},d,m}$ from Theorem \ref{th:indecomp}.  Recall that the conditions of Theorem 1 give us information about the number $d$, as well as the terms of the vector $\mathbf{a} = (a_0,a_1,\cdots,a_{m-1})$.  The hypotheses are simpler for odd $p$, but when $p=2$ there are certain technicalities that emerge which need to be accounted for.  To facilitate working with these hypotheses, we will list them here:
    \begin{enumerate}
        \renewcommand{\theenumi}{\Roman{enumi}}
       \item\label{it:exc.mod.indecom.condition...d.in.U1} $d\in U_1$
        \item\label{it:exc.mod.indecom.condition...power.of.d} $d^{p^{a_i}}\in U_{i+1}$ for all $0\le i<m$
        \item\label{it:exc.mod.indecom.condition...ai.vs.aj.inequality}
          $a_i+j<a_{i+j}$ for all $0\le i<m$ with  $1\le j<
          m-i$ and $a_{i+j}\neq -\infty$, except if $p=2$, $d\not\in
          U_2$, $i=0$, and $a_i=0$, in which case $a_j\neq 0$ for all
          $1\le j<m$
        \item\label{it:exc.mod.indecom.condition...a0.bounds} if $p=2$ and $n=1$, then $a_0 = -\infty$
        \item\label{it:exc.mod.indecom.condition...ai.vs.v.inequality}
          If $p=2$, $m\ge 2$, $d\in -U_v\setminus -U_{v+1}$ for
          some $v\ge 2$, and $a_0=0$, then $a_{i}>i-(v-1)$ for
          all $v\le i<m$ and $a_i\neq -\infty$.
    \end{enumerate}
So, for example, if we refer to condition (\ref{it:exc.mod.indecom.condition...a0.bounds}) of Theorem 1, we simply mean the hypothesis that $a_0=-\infty$ in the case where $p=2$ and $n=1$.

Our overall strategy is as follows: we will first prove the results in some special cases.  As we move toward the general case, we will need to understand how the invariants attached to $X_{\mathbf{a},d,m}$ manifest relative to properties of particular elements in the module.  These results --- Lemmas \ref{le:indecomp1b}, \ref{le:indecomp2}, and \ref{le:indecomp3} --- are relatively technical in their nature, and the reader might choose to skip forward to the Proof of Theorem \ref{th:indecomp} on a first read and take them on faith.

To get started, we begin with a number of observations. 

First, for each $0 \leq i \leq m-1$, the relations on $X_{\mathbf{a},d,m}$ imply that $l(x_i) = p^{a_i}$, and moreover we have $l(y) = p^{a_0}+1$.  In particular, this implies that $a_0<n$.

Second, note that the conditions (\ref{it:exc.mod.indecom.condition...d.in.U1})--(\ref{it:exc.mod.indecom.condition...ai.vs.v.inequality}) of
Theorem \ref{th:indecomp}\index{Theorem \ref{th:indecomp}} are satisfied for $\mathbf{a}=(a_0,\dots,a_{m-1})$, $d$, and $m$, then they are satisfied for
$(a_0,\dots,a_{i-1})$, $d$, and $i$ for $1\le i<m$ as well.

Third, condition (\ref{it:exc.mod.indecom.condition...ai.vs.aj.inequality}) may be used to conclude that for all $0\le i<j<m$ with $a_i, a_j\neq -\infty$, one has $a_i<a_j$.  We will occasionally use this weaker form (referring to it as (\ref{it:exc.mod.indecom.condition...ai.vs.aj.inequality}) as well), and it is valid in all cases.

Finally, we remark that because $X_{\mathbf{a},d,m}$ is a finitely generated module of $R_mG$, it can be decomposed into a direct sum of indecomposable submodules in a unique way.  
This is a consequence of the Krull-Schmidt theorem.


We start our investigation into the indecomposability of $X_{\mathbf{a},d,m}$ by examining two special cases.

\begin{lemma}\label{le:casea}\index{Lemma \ref{le:casea}}
    Assume the hypotheses of Theorem \ref{th:indecomp}\index{Theorem \ref{th:indecomp}} and
    let $m\ge 2$.  Then in each of the following two cases, $X$ is an
    indecomposable $R_mG$-module:
    \begin{enumerate}
      \item\label{it:casea1} $a_{m-1}=0$, or
      \item\label{it:casea2} $p=2$, $d\not\in U_2$, $a_0=0$,
        $a_{m-1}=1$.
    \end{enumerate}
\end{lemma}

\begin{proof}
    In the first case, by (\ref{it:exc.mod.indecom.condition...ai.vs.aj.inequality}) we have $a_{i}=-\infty$ for
    $0\le i<m-1$.  The same is true the second case for $0<i<m-1$, again by (\ref{it:exc.mod.indecom.condition...ai.vs.aj.inequality}) and (\ref{it:exc.mod.indecom.condition...ai.vs.v.inequality}).

    In either case, suppose that $X_{\mathbf{a},d,m}\simeq V\oplus W$ with
    $V, W \neq \{0\}$. By Nakayama's Lemma, $V/p^{m-1}V$ and
    $W/p^{m-1}W$ are nonzero as well, and our first goal is to
    determine the isomorphism classes of $V/p^{m-1}V$ and
    $W/p^{m-1}W$.

    In view of the defining relations of $X_{\mathbf{a},d,m}$, we have
    \begin{equation*}
        X/p^{m-1}X = A\oplus B,
    \end{equation*}
    where
    \begin{equation*}
        A := X_{\check{\mathbf{a}},d,m-1}
    \end{equation*}
    with $\check{\mathbf{a}} = (a_0,\dots,a_{m-2})$ and
    \begin{equation*}
        B := \langle x_{m-1}\rangle/(p^{m-1}x_{m-1},
        (\sigma^{p^{a_{m-1}}}-1)x_{m-1})\simeq R_{m-1}G_{a_{m-1}}.
    \end{equation*}

    We claim that each of $V$ and $W$ is cyclic, as follows.
    In case (\ref{it:casea1}),
    \begin{equation*}
        X_{\check{\mathbf{a}},d,m-1} = \langle y : (\sigma-d)y=0 \rangle
    \end{equation*}
    while in case (\ref{it:casea2}),
    \begin{equation*}
        X_{\check{\mathbf{a}},d,m-1} = \langle y,x_0 : (\sigma-d)y=x_0,
        (\sigma-1)x_0=0 \rangle.
    \end{equation*}
    In both cases we deduce that $A$ is cyclic.  We see that $B$ is
    cyclic by definition. By the Krull-Schmidt Theorem, all decompositions of $X/p^{m-1}X$
    into indecomposable $R_{m-1}G$-modules are equivalent. Hence
    $V/p^{m-1}V\oplus W/p^{m-1}W$ decomposes into precisely two
    indecomposable $R_{m-1}G$-modules.  Hence we may assume without
    loss of generality that
    \begin{align*}
        V/p^{m-1}V &\simeq A \\
        W/p^{m-1}W &\simeq B \simeq R_{m-1}G_{a_{m-1}}.
    \end{align*}
    By Lemma~\ref{le:cycprop}\index{Lemma \ref{le:cycprop}}, $V$ and $W$ are cyclic and we let $v\in
    V$, $w\in W$ denote generators.  If $v\in (p,\sigma-1)X$, since
    $X=V\oplus W$ we then obtain $v\in (p,\sigma-1)V$. By Nakayama's Lemma we conclude
    $V=\{0\}$, a contradiction.  Hence $v, w\not\in (p,\sigma-1)X$.

    (\ref{it:casea1}).  Suppose $a_{m-1}=0$.  By (\ref{it:exc.mod.indecom.condition...power.of.d}), we then have
    $d=d^{p^0} = d^{p^{a_{m-1}}}\in U_m$.  But this means that $d \equiv 1 \pmod{p^m}$. Since $d$ acts on $X$ as an element of $R_m \subseteq R_mG$, it is only defined up to a multiple of $p^m$, and hence we may assume without loss of
    generality that $d=1$.  Therefore
    \begin{align*}
        X :&= \langle y, x_{m-1} : (\sigma-1)y=p^{m-1}x_{m-1},
        \sigma x_{m-1}=x_{m-1} \rangle \\
        &= \langle y, x_{m-1} : \sigma y = y + p^{m-1}x_{m-1},
        \sigma x_{m-1}=x_{m-1}\rangle,
    \end{align*}
    and we obtain that each element of $X$ may be written
    $ay+bx_{m-1}$ for $a, b\in R_m$.  In particular, $v$ and $w$
    may be written in this form.  Since $v,w \not\in pX$, we cannot
    have that $a, b\in \langle p\rangle$.  Hence at least one of $a$
    and $b$ is a unit in $R_m$.  Without loss of generality we may
    assume that each of $v$ and $w$ may be written in the form
    $ay+bx_{m-1}$ with at least one of $a$ or $b$ equal to $1$.

    We claim that $p^{m-1}x_{m-1}\neq 0$ in $X$, as follows.
    Suppose $p^{m-1}x_{m-1}=0$; this means that $p^{m-1}x_{m-1}$ is some combination of the relations that define $X_{\mathbf{a},d,m}$ in equation (\ref{eq:definition.of.X}).  Hence there exist $f_y, f_{m-1}\in R_m
    G$ such that in $R_mG[y,x_{m-1}]$ we have
    \begin{equation*}
        p^{m-1}x_{m-1} = f_y\left( (\sigma-1)y-p^{m-1}x_{m-1} \right)
            + f_{m-1}(\sigma-1)x_{m-1}.
    \end{equation*}
    Considering the coefficients of $y$ in the equation above, we
    obtain that $f_y\in \ker_{R_mG} (\sigma-1)$.  By
    Lemma~\ref{le:phidb}\index{Lemma \ref{le:phidb}} we have $f_y\in \langle P(n,0)\rangle$, say
    $f_y=g_yP(n,0)$ for $g_y\in R_m G$.  Now consider the
    coefficients of $x_{m-1}$ in the equation above:
    \begin{equation*}
        p^{m-1} = -p^{m-1} g_y P(n,0) + f_{m-1}(\sigma-1).
    \end{equation*}
    Let $\chi_0:R_mG \to R_m$ be the natural $R_m$-homomorphism
    sending $\sigma$ to $1$.  We see that $\chi_0(P(n,0))=p^n$, and
    since $m\ge 2$, we have $\chi_0(p^{m-1}P(n,0))=0$.
    Applying $\chi_0$ to our equation above, we obtain the equation
    in $R_m$
    \begin{equation*}
        p^{m-1} = 0,
    \end{equation*}
    a contradiction.  Hence $p^{m-1}x_{m-1}\neq 0$ in $X$, as desired.

    Now observe that
    \begin{align*}
       (\sigma-1)(y+bx_{m-1}) &= p^{m-1}x_{m-1}, \quad b\in R_m\\
       p^{m-1}(ay+x_{m-1}) &= p^{m-1}x_{m-1}, \quad a\in \langle
       p\rangle.
    \end{align*}
    Now each of $v$ and $w$ take the form either of $y+bx_{m-1}$ or,
    if the coefficient of $y$ is not a unit, $ay+x_{m-1}$, $a\in
    \langle p\rangle = R_m\setminus U(R_m)$.  Hence $0 \neq p^{m-1} x_{m-1} \in W\cap V$, contradicting the fact that $X=W\oplus V$. Therefore $X$ is indecomposable.

    (\ref{it:casea2}).  Suppose $p=2$, $d\not\in U_2$, $a_0=0$, and
    $a_{m-1}=1$.  By condition (\ref{it:exc.mod.indecom.condition...a0.bounds}) we must have $n>1$.  By (\ref{it:exc.mod.indecom.condition...power.of.d}), $d^2=d^{2^{a_{m-1}}}\in U_m$.  By
    Lemma~\ref{le:upower}\index{Lemma \ref{le:upower}}, we deduce that either $d=-1$ or $d\in -U_v$
    with $v\ge m-1$. By (\ref{it:exc.mod.indecom.condition...ai.vs.v.inequality}), if $v=m-1$, then
    $a_{m-1}>m-1-(v-1)=1$, a contradiction.  Hence $v\ge m$.  Since
    the action of $d$ in $X$ is defined only up to $2^m$, we may
    assume without loss of generality that $d=-1$.  Hence
    \begin{equation}\label{eq:relations.in.case.2}
        X := \langle y, x_0, x_{m-1} :\ (\sigma+1)y=x_0+2^{m-1}x_{m-1},
        \sigma x_0 = x_0, \sigma^2 x_{m-1}=x_{m-1} \rangle
    \end{equation}
    and we obtain that each element of $X$ may be written
    $ay+bx_{m-1}$ for $a, b\in R_m G$.  In particular, $v$ and $w$
    may be written in this form.  Since $v,w \not\in (2,\sigma-1)X$, we cannot
    have that $a, b\in (2,\sigma-1)R_mG$.  Hence at least one of $a$
    and $b$ is a unit in $R_mG$.  Without loss of generality we may
    assume that each of $v$ and $w$ may be written in the form
    $ay+bx_{m-1}$ with at least one of $a$ or $b$ equal to $1$.

    We claim that $2^{m-1}(\sigma+1)x_{m-1}\neq 0$ in $X$, as follows.
    Suppose $2^{m-1}(\sigma+1)x_{m-1}=0$.  Then there exist $f_y, f_0,
    f_{m-1}\in R_m
    G$ such that in $R_mG[y,x_0,x_{m-1}]$ we have
    \begin{multline}\label{eq:casea21}
        2^{m-1}(\sigma+1)x_{m-1} = f_y\left( (\sigma+1)y-x_0-2^{m-1}x_{m-1} \right)
             + \\ f_0(\sigma-1)x_0 + f_{m-1}(\sigma^2-1)x_{m-1}.
    \end{multline}
    Considering the coefficients of $y$ in equation~\eqref{eq:casea21}
    above, we
    obtain that $f_y\in \ker_{R_mG} (\sigma+1)$.  Since $d^2\in U_m$, by
    Lemma~\ref{le:phidb}\index{Lemma \ref{le:phidb}}, we have $f_y\in \langle Q_{-1}(n,0)\rangle$, say
    $f_y=g_yQ_{-1}(n,0)$ for $g_y\in R_m G$.  Let $\chi_1:R_m G\to
    R_mG_1$
    be the natural $R_m$-homomorphism sending $\sigma^2\to 1$.  Now by
    Lemma~\ref{le:qhomo}\index{Lemma \ref{le:qhomo}} we have $\chi_1(Q_d(n,0)) = 0 \mod 2^{n-1} R_m G_1$, and
    since $m\ge 2$ and $n>1$, we get $\chi_1(2^{m-1}Q_d(n,0)) = 0$.
    Considering the coefficients of $x_{m-1}$ in
    equation~\eqref{eq:casea21} and applying $\chi_1$, we obtain the
    equation in $R_m G_1$
    \begin{equation*}
        2^{m-1}(\sigma+1) = 0,
    \end{equation*}
    a contradiction.  Hence $2^{m-1}(\sigma+1)x_{m-1}\neq 0$ in $X$, as desired.

    Now observe that for any $b\in R_mG$,
    \begin{align*}
       (\sigma-1)(\sigma+1)(y+bx_{m-1}) &=
       (\sigma-1)(x_0+2^{m-1}x_{m-1}+b(\sigma+1)x_{m-1})\\
       &= 2^{m-1}(\sigma-1)x_{m-1} \\ &= 2^{m-1}(\sigma+1)x_{m-1}.
    \end{align*}
    Moreover, for any $a\in \langle 2,\sigma-1\rangle$,
    \begin{align*}
       2^{m-1}(\sigma+1)(x_{m-1}+ay) &= 2^{m-1}(\sigma+1)x_{m-1} +
       a2^{m-1}(x_0+2^{m-1}x_{m-1})\\
       &= 2^{m-1}(\sigma+1)x_{m-1},
    \end{align*}
    since $m\ge 2$.  Now each of $v$ and $w$ take the form either of
    $y+bx_{m-1}$ or, if the coefficient of $y$ is not a unit,
    $ay+x_{m-1}$, $a\in \langle 2, \sigma-1\rangle = R_mG\setminus
    U(R_mG)$.  Hence $0 \neq 2^{m-1}(\sigma+1)x_{m-1} \in W \cap V$,
    contradicting the fact that $X=W\oplus V$.  Therefore $X$ is
    indecomposable.
\end{proof}

Now we develop some machinery to handle the cases not considered in
the previous lemma.

\begin{lemma}\label{le:indecomp1}\index{Lemma \ref{le:indecomp1}}
    Assume the hypotheses of Theorem \ref{th:indecomp}\index{Theorem \ref{th:indecomp}} and let
    $m\ge 2$. Suppose that $a_{m-1}\neq -\infty$ and for some $v\in
    X$,
    \begin{equation*}
        (\sigma^{p^{a_{m-1}}}-1)v = p^{m-1}\sum_{i=0}^{m-1}
        g_ix_i,\quad g_i\in R_mG.
    \end{equation*}
    Then $p^{m-1}g_{m-1}\in \langle p^{m-1}(\sigma-1)^{p^{a_{m-1}}-1}
    \rangle$.
\end{lemma}

\begin{proof}
    By hypothesis (\ref{it:exc.mod.indecom.condition...ai.vs.aj.inequality}) we have $a_i\le a_{m-1}$ for all
    $0\le i\le m-1$.  Therefore in $R_mG$ we have
    $(\sigma^{p^{a_i}}-1) P(a_{m-1},a_i)= (\sigma^{p^{a_{m-1}}}-1)$.
    From the defining relations for $X$ we deduce that
    $(\sigma^{p^{a_{m-1}}} -1)x_i = 0$ for all $0\le i\le m-1$.
    Hence without loss of generality we may assume that $v=ry$ for
    $r\in R_mG$.

    Suppose then that $(\sigma^{p^{a_{m-1}}}-1)ry =
    p^{m-1}\sum_{i=0}^{m-1} g_ix_i$.  By the defining relations for
    $X$ we have
    \begin{equation}\label{eq:spec0}
        (\sigma^{p^{a_{m-1}}}-1)ry - p^{m-1}\sum_{i=0}^{m-1} g_ix_i =
        \\ f_y((\sigma-d)y-\sum_{i=0}^{m-1} p^ix_i) +
        \sum_{i=0}^{m-1} f_i(\sigma^{p^{a_i}}-1)x_i
    \end{equation}
    for some $f_y$ and $f_i$, $i=0$, $\dots,$ $m-1$, in $R_mG$.

    Consider first the coefficients of $x_{m-1}$ in
    \eqref{eq:spec0}. We have in $R_mG$
    \begin{equation*}
        -p^{m-1}g_{m-1}=-p^{m-1}f_y+f_{m-1}(\sigma^{p^{a_{m-1}}}-1).
    \end{equation*}
    Now let $\chi_{a_{m-1}}: R_m G\to R_mG_{a_{m-1}}$ be the natural
    $R_m$-homomorphism. We deduce that in $R_mG_{a_{m-1}}$,
    \begin{equation*}
        -p^{m-1}\chi_{a_{m-1}}(g_{m-1})
        =-p^{m-1}\chi_{a_{m-1}}(f_y).
    \end{equation*}
    By Lemma~\ref{le:kerbasic}\index{Lemma \ref{le:kerbasic}},
    $\chi_{a_{m-1}}(g_{m-1})=\chi_{a_{m-1}}(f_y) \mod
    pR_mG_{a_{m-1}}$.  Hence
    \begin{equation*}
        g_{m-1}=f_y \mod (p,\sigma^{p^{a_{m-1}}}-1)R_mG.
    \end{equation*}

    Now consider the coefficients of $y$ in \eqref{eq:spec0}:
    \begin{equation*}
        (\sigma^{p^{a_{m-1}}}-1)r = f_y(\sigma-d).
    \end{equation*}
    Applying $\chi_{a_{m-1}}$ to this equation, we have in
    $R_mG_{a_{m-1}}$
    \begin{equation*}
        0 = \chi_{a_{m-1}}(f_y)(\sigma-d).
    \end{equation*}

Since $a_{m-1}\neq -\infty$,  condition (\ref{it:exc.mod.indecom.condition...power.of.d})
    gives $d^{p^{a_{m-1}}}\in U_m$. By Lemma~\ref{le:phidb}\index{Lemma \ref{le:phidb}}, we have
    $f_y\in (Q_d(a_{m-1},0),\sigma^{p^{a_{m-1}}}-1)R_mG$. Observe
    that since $d\in U_1$ by hypothesis (\ref{it:exc.mod.indecom.condition...d.in.U1}), we get
    $Q_d(a_{m-1},0)=P(a_{m-1},0) \mod p\Z G$.  Hence $g_{m-1}\in
    (p,P(a_{m-1},0),\sigma^{p^{a_{m-1}}}-1)R_mG$ as well.

    Recalling that in $R_mG/pR_mG\simeq \Fp G$ we have
    $P(a_{m-1},0)=(\sigma-1)^{p^{a_{m-1}}-1}$ and
    $(\sigma^{p^{a_{m-1}}}-1)=(\sigma-1)^{p^{a_{m-1}}}$, we have
    obtained our result:
    \begin{equation*}
        p^{m-1}g_{m-1} \in \langle p^{m-1}(\sigma-1)^{
        p^{a_{m-1}}-1}\rangle \subset R_mG.
    \end{equation*}
\end{proof}

Recall that from the relations for $X_{\mathbf{a}, d, m}$ we have that
$l(y)=p^{a_0}+1$ and $l(x_i)=p^{a_i}$ for $i=0, \dots, m-1$.  In
what follows we will have frequent occasion to compare $l(y)$ to
$l(w)$ for $w\in X_{\mathbf{a},d,m}$ satisfying $l(w)=p^c$ for some $c$.
We observe in particular that if $l(y)<l(w)$ for such a $w$, then
since $1\le p^{a_0}+1$, we must have $c>0$ and therefore $l(w)\ge 2$.

\begin{lemma}\label{le:indecomp1b}\index{Lemma \ref{le:indecomp1b}}
    Assume the hypotheses of Theorem \ref{th:indecomp}\index{Theorem \ref{th:indecomp}} and let
    $m\ge 2$. Suppose that $w\in X_{\mathbf{a},d,m}$ satisfies $l(y)<l(w)=p^{a_{m-1}}$.
    Then
    \begin{enumerate}
        \item\label{it:ind1b1} $l(x_i)\le l(w)-2$ for all $0\le
        i<m-1$
        \item\label{it:ind1b2} $l(y)\le l(w)-2$.
    \end{enumerate}
\end{lemma}

\begin{proof}\

    (\ref{it:ind1b1}). First, from hypothesis (\ref{it:exc.mod.indecom.condition...ai.vs.aj.inequality}) we
    conclude that $a_i<a_{m-1}$ and hence
    $l(x_i)=p^{a_i}<l(w)=p^{a_{m-1}}$.  If $p^{a_i}> l(w)-2$, then we
    must have that $p^{a_i}=p^{a_{m-1}}-1$. But since
    $p^{a_{m-1}}\ge 2$ we conclude that $p=2$, $a_i=0$, and
    $a_{m-1}=1$.  By (\ref{it:exc.mod.indecom.condition...ai.vs.aj.inequality}), $a_i+1<a_{m-1}$ for $m-1>i$,
    so that we cannot have $a_i=0$ and $a_{m-1}=1$ for $i<m-1$,
    unless $p=2$ and $i=0$.  Then $l(y)=2^{a_0}+1=2$ and
    $l(w)=2^{a_{m-1}}=2$, contradicting $l(y)<l(w)$.
    Hence in all cases $l(x_i)\le l(w)-2$ for all $0\le i<m-1$.

    (\ref{it:ind1b2}). Since $l(y)=p^{a_0}+1 $ and we are in the
    case $l(y)<l(w)$, it remains to exclude the case
    $p^{a_0}+1=l(w)-1$.  In this case $p^{a_0}+2=p^{a_{m-1}}$.  Then
    one of the following cases occurs: $p=2$,
    $a_0=1$, $a_{m-1}=2$, contrary to (\ref{it:exc.mod.indecom.condition...ai.vs.aj.inequality}); or $p=3$, $a_0=0$,
    and $a_{m-1}=1$, again contrary to (\ref{it:exc.mod.indecom.condition...ai.vs.aj.inequality}).
\end{proof}

\begin{lemma}\label{le:indecomp2}\index{Lemma \ref{le:indecomp2}}
    Assume the hypotheses of Theorem \ref{th:indecomp}\index{Theorem \ref{th:indecomp}} and let $m\ge 2$.
    Suppose that $w\in X_{\mathbf{a},d,m}$ satisfies the following:
    \begin{itemize}
        \item $W=\langle w\rangle$ is a direct summand of $X_{\mathbf{a},d,m}$
        \item $W/p^{m-1}W$ is a free $R_{m-1} G_{a_{m-1}}$-module
        \item $l(y)<l(w)$.
    \end{itemize}
    Then
    \begin{enumerate}
        \renewcommand{\theenumi}{\alph{enumi}}
        \item\label{it:j1} $w = g_yy+\sum_{i=0}^{m-1}g_ix_i$ with
        $g_{m-1}\in U(R_mG)$, $g_i\in R_mG$ for $0\le i<m-1$, and
        $g_y\in R_m$
        \item\label{it:j2} $p^{m-1}(\sigma-1)w\in p^{m-1}\langle
        x_0,(\sigma-1)x_1,\dots,(\sigma-1)x_{m-1}\rangle$
        \item\label{it:j3} $(\sigma^{p^{a_{m-1}}}-1)w \in
        p^{m-1}\langle
        x_0,\dots,x_{m-2},(\sigma-1)^{p^{a_{m-1}}-1}x_{m-1}\rangle$
        \item\label{it:j4} $\langle
        p^{m-1}(\sigma-1)^{l(w)-2}w\rangle = \langle
        p^{m-1}(\sigma-1)^{l(w)-2}x_{m-1}\rangle \neq \{0\}$.
    \end{enumerate}
\end{lemma}

\begin{proof}
    Since $W/p^{m-1}W$ is a free $R_{m-1}G_{a_{m-1}}$-module, $W/pW$
    is a free $\Fp G_{a_{m-1}}$-module.  Since $W$ is a direct
    summand, $l_{X_{\mathbf{a},d,m}}(w)=l_W(w)=p^{a_{m-1}}$.

    (\ref{it:j1}). Let $w=g_yy+\sum_{i=0}^{m-1}g_ix_i$ with $g_y$
    and $g_i$, $i=0, 1, \dots, m-1$ elements of $R_mG$.  Since
    $(\sigma-d)y = \sum_{i=0}^{m-1} p^ix_i$, we may assume that
    $g_y\in R_m$.

    Suppose that $g_{m-1}\not\in U(R_m G)$. Then $g_{m-1}\in \langle
    (\sigma-1),p\rangle$. Multiplying $w$ by $(\sigma-1)^{l(w)-1}$
    and using Lemma~\ref{le:indecomp1b}\index{Lemma \ref{le:indecomp1b}}, we obtain
    \begin{align*}
       (\sigma-1)^{l(w)-1}w &= (\sigma-1)^{l(w)-1}g_{m-1}x_{m-1}
       &\mod pX\\ &\in (\sigma-1)^{p^{a_{m-1}}-1}\langle
       (\sigma-1,p)\rangle x_{m-1} &\mod pX\\ &= 0 &\mod pX
    \end{align*}
    contrary to the value of $l_X(w)$.  Hence $g_{m-1}\in U(R_m
    G)$ and we have established (\ref{it:j1}).

    (\ref{it:j2}). Consider $p^{m-1}(\sigma-1)w$.  Write $w=g_y y +
    \sum_{i=0}^{m-1}g_ix_i$ as above, and observe that since $d\in
    U_1$ by (\ref{it:exc.mod.indecom.condition...d.in.U1}), $p^{m-1}(\sigma-1)=p^{m-1}(\sigma-d)$ in
    $R_mG$.  Then in $R_m G$ we have
    \begin{equation*}
        p^{m-1}(\sigma-1)w = p^{m-1}\left(
        g_y(\sigma-d)y + \sum_{i=0}^{m-1}g_i(\sigma-1)x_i\right).
    \end{equation*}
    Using the relation $(\sigma-d)y=\sum_{i=0}^{m-1}p^ix_i$, we
    deduce that in $R_mG$
    \begin{align*}
        p^{m-1}(\sigma-1)w &= p^{m-1}\left(
        g_y\left(\sum_{i=0}^{m-1}p^ix_i\right) +
        \sum_{i=0}^{m-1}g_i(\sigma-1)x_i\right) \\
        &= \left( g_y\sum_{i=0}^{m-1} p^{m-1+i}x_i\right)
        + p^{m-1}\sum_{i=0}^{m-1} g_i(\sigma-1)x_i \\
        &= p^{m-1}\left(g_y x_0+\sum_{i=0}^{m-1}g_i(\sigma-1)x_i\right)
        \\ &\in p^{m-1}\langle x_0,(\sigma-1)x_1,\dots,(\sigma-1)x_{m-1}
        \rangle,
    \end{align*}
    as desired.

    (\ref{it:j3}). Since $w$ is a generator for $W$ and $W/p^{m-1}W$
    is a free $R_{m-1} G_{a_{m-1}}$-module,
    \begin{equation*}
      (\sigma^{p^{a_{m-1}}}-1)w = p^{m-1}fw
    \end{equation*}
    for some $f\in R_mG$. Writing $w=g_y y +
    \sum_{i=0}^{m-1}g_ix_i$ as in part (\ref{it:j1}), we obtain from the
    defining relations of $X$ that
    \begin{equation}
        \left((\sigma^{p^{a_{m-1}}}-1)-p^{m-1}f\right)\left(g_yy +
        \sum_{i=0}^{m-1} g_ix_i\right) = \\
        f_y\left((\sigma-d)y-\sum_{i=0}^{m-1}p^ix_i\right)+ \sum_{i=0}^{m-1}
        f_i(\sigma^{p^{a_i}}-1)x_i
    \end{equation}
    for $f_y$ and $f_i$, $i=0, 1, \dots, m-1$, elements of $R_mG$.

    Consider the coefficient of $y$:
    \begin{equation*}
      (\sigma^{p^{a_{m-1}}}-1)g_y-p^{m-1}fg_y = f_y(\sigma-d).
    \end{equation*}

    Now by (\ref{it:exc.mod.indecom.condition...power.of.d}), $d^{p^{a_{m-1}}} \in U_m$. Since $a_{m-1}\le
    n$, we get $d^{p^{n}}\in U_m$.
    Applying $\phi_d^{(m)}$ to the previous equation, we deduce that in
    $R_m$
    \begin{equation*}
        -p^{m-1}\phi_d^{(m)}(f)\phi_d^{(m)}(g_y)=0.
    \end{equation*}
    Hence either $\phi_d(f)=0 \mod p$ or $\phi_d(g_y)=0\mod p$.

    If $\phi_d(f)=0\mod p$, then
    \begin{equation*}
        f=pe+(\sigma-d)g
    \end{equation*}
    for some $e, g\in R_mG$.  Observe that since $d\in U_1$ by
    (\ref{it:exc.mod.indecom.condition...d.in.U1}), $p^{m-1}(\sigma-d) = p^{m-1}(\sigma-1)$.
    Hence we obtain
    \begin{equation*}
        (\sigma^{p^{a_{m-1}}}-1)w = p^{m-1}fw = gp^{m-1}(\sigma-1)w.
    \end{equation*}
    Applying part (\ref{it:j2}), we obtain
    \begin{equation*}
        (\sigma^{p^{a_{m-1}}}-1)w \in p^{m-1}
        \langle x_0,\dots,x_{m-2},(\sigma-1)x_{m-1}\rangle.
    \end{equation*}
    Then, by Lemma~\ref{le:indecomp1}\index{Lemma \ref{le:indecomp1}},
    \begin{equation*}
        (\sigma^{p^{a_{m-1}}}-1)w \in p^{m-1}
        \langle x_0, \dots,x_{m-2}, (\sigma-1)^{p^{a_{m-1}}-1}
        x_{m-1}\rangle.
    \end{equation*}

    If, on the other hand, $\phi_d(g_y)=0\mod p$, then
    \begin{equation*}
        g_y=pe+(\sigma-d)g
    \end{equation*}
    for some $e, g\in R_mG$.  Because $(\sigma-d) y =
    \sum_{i=0}^{m-1} p^ix_i$,
    \begin{equation*}
        (\sigma^{p^{a_{m-1}}}-1)w = p^{m-1}f(g_yy + \sum_{i=0}^{m-1}
        g_ix_i) \in p^{m-1}\langle x_0,\dots,x_{m-1}\rangle.
    \end{equation*}
    By Lemma~\ref{le:indecomp1}\index{Lemma \ref{le:indecomp1}},
    \begin{equation*}
        (\sigma^{p^{a_{m-1}}}-1)w \in p^{m-1}\langle
        x_0,\dots,x_{m-2}, (\sigma-1)^{p^{a_{m-1}}-1}x_{m-1}\rangle.
    \end{equation*}
    In either case, then, we have established part (\ref{it:j3}).

    (\ref{it:j4}). Now for any $x\in X$, $p^{m-1}(\sigma-1)^{l(x)}x=0$.  Hence applying $(\sigma-1)^{l(w)-2}$ to our expression $w = g_yy+\sum_{i=0}^{m-1}g_ix_i$, and recalling Lemma~\ref{le:indecomp1b}\index{Lemma \ref{le:indecomp1b}}, shows $$p^{m-1}(\sigma-1)^{l(w)-2}w = p^{m-1}(\sigma-1)^{l(w)-2}g_{m-1}x_{m-1}.$$  Therefore we may assume without loss of generality that $g_i=0$ for all $0\le i<m-1$ and $g_y=0$ as well.  Then $w=g_{m-1}x_{m-1}$.

    Suppose that $p^{m-1}(\sigma-1)^{l(w)-2}w =0$. By multiplying by
    a unit of $R_mG$ we may assume that $g_{m-1}=1$.

    From the defining relations of $X$ we deduce
    \begin{equation}\label{eq:indecomeq1}
        p^{m-1}(\sigma-1)^{l(w)-2}x_{m-1}=\\ f_y((\sigma-d)y-
        \sum_{i=0}^{m-1}
        p^ix_i)+\sum_{i=0}^{m-1} f_i(\sigma^{p^{a_i}}-1)x_i
    \end{equation}
    for $f_y\in R_mG$ and $f_i$, $i=0, 1,\dots, m-1$, elements of
    $R_mG$.

    Consider the coefficients of $x_{m-1}$ in \eqref{eq:indecomeq1}:
    \begin{equation*}
      p^{m-1}(\sigma-1)^{l(w)-2}=-p^{m-1}f_y +
      f_{m-1}(\sigma^{p^{a_{m-1}}}-1).
    \end{equation*}
    Now apply the natural $R_m$-homomorphism $\chi_{a_{m-1}}:R_m G
    \to R_m G_{a_{m-1}}$. Then we have
    \begin{equation*}
        p^{m-1}\chi_{a_{m-1}}(\sigma-1)^{l(w)-2}=
        -p^{m-1}\chi_{a_{m-1}}(f_y).
    \end{equation*}
    Now $\ann_{R_m G_{a_{m-1}}} p^{m-1} = pR_{m} G_{a_{m-1}}$ by
    Lemma~\ref{le:kerbasic}\index{Lemma \ref{le:kerbasic}}, so we have for a suitable element $h_y\in
    R_m G$
    \begin{equation*}
        \chi_{a_{m-1}}(f_y) = -\chi_{a_{m-1}}(\sigma-1)^{l(w)-2} +
        p\chi_{a_{m-1}}(h_y).
    \end{equation*}
    Therefore
    \begin{equation}\label{eq:spec5}
        f_y = -(\sigma-1)^{l(w)-2} +p h_y +
        r_y(\sigma^{p^{a_{m-1}}}-1), \quad r_y\in R_m G.
    \end{equation}

    Now consider the coefficients of $y$ in \eqref{eq:indecomeq1}:
    \begin{equation*}
        0 = f_y(\sigma-d).
    \end{equation*}
    Since by (\ref{it:exc.mod.indecom.condition...d.in.U1}), $d\in U_1$, we have $0=(\sigma-1)f_y
    \mod pR_m G$. Putting this result together with \eqref{eq:spec5}
    and the fact that $(\sigma^{p^{a_{m-1}}}-1) =
    (\sigma-1)^{p^{a_{m-1}}}$ modulo $pR_mG$, we obtain
    \begin{equation*}
        0 = -(\sigma-1)^{l(w)-1} + r_y(\sigma-1)^{p^{a_{m-1}}+1} \mod
        pR_m G.
    \end{equation*}

    Observing that $R_m G/pR_m G\simeq \Fp G$ and that, in $\Fp G$,
    $(\sigma-1)^{l(w)-1}\notin \langle
    (\sigma-1)^{p^{a_{m-1}}+1}\rangle$ since
    $0<l(w)-1<p^{a_{m-1}}\le p^n$, we have reached a contradiction.
    Hence
    \begin{equation*}
        p^{m-1}(\sigma-1)^{p^{a_{m-1}}-2}w \neq 0.
    \end{equation*}
    Moreover, since $w=g_{m-1}x_{m-1}$ for $g_{m-1}\in U(R_mG)$,
    $\langle w\rangle = \langle x_{m-1}\rangle$ and we
    have
    \begin{equation*}
        \langle p^{m-1}(\sigma-1)^{l(w)-2}w \rangle =
        \langle p^{m-1}(\sigma-1)^{l(w)-2}x_{m-1} \rangle
        \neq \{0\},
    \end{equation*}
    as desired.
\end{proof}

\begin{lemma}\label{le:indecomp3}\index{Lemma \ref{le:indecomp3}}
    Assume the hypotheses of Theorem \ref{th:indecomp}\index{Theorem \ref{th:indecomp}} and let
    $m\ge 2$.  Suppose that $w\in X$ satisfies the following:
    \begin{itemize}
        \item $W=\langle w\rangle$ is a direct summand of $X$
        \item $W/p^{m-1}W$ is a free $R_{m-1}G_{a_{m-1}}$-module
        \item $l(y)<l(w)$
        \item $a_i\neq -\infty$ for some $0\le i<m-1$.
    \end{itemize}
    Suppose $\proj_{W} y=g_ww$ for some $g_w\in R_mG$.  Let
    \begin{equation*}
        c=\max\{a_i : 0\le i<m-1\}.
    \end{equation*}
    Then
    \begin{equation*}
        (\sigma^{p^{c}}-1)(\sigma-d)g_w
        \in \langle p^{m-1}(\sigma-d)(\sigma-1)^{p^c},
        (\sigma^{p^{a_{m-1}}}-1)\rangle.
    \end{equation*}
\end{lemma}

\begin{proof}
    Let $S=\{i : 0\le i\le m-2,\ a_i\neq -\infty\}$ and $s$ be the
    cardinality of $S$.  By hypothesis, $s\ge 1$ and by (\ref{it:exc.mod.indecom.condition...ai.vs.aj.inequality}),
    if $j<i$ are elements of $S$, we have $a_j<a_i$.

    We begin with the relation
    \begin{equation*}
        (\sigma-d)y = \sum_{i=0}^{m-2} p^ix_i \mod p^{m-1}X
    \end{equation*}
    and approach the conditions of Lemma~\ref{le:tech}\index{Lemma \ref{le:tech}}.

    Suppose that $j<i$ are elements of $S$.  Then since
    $(\sigma^{p^{a_j}}-1) P(a_i,a_j)=(\sigma^{p^{a_i}}-1)$ in $R_{m}G$,
    $(\sigma^{p^{a_i}}-1)x_j=0$ for all $j\le i$ elements of
    $S$.  Now for $i\in S$ set
    \begin{equation*}
        e_i:=(m-1)-\min \{j\in S\cup \{m-1\} : i<j\}.
    \end{equation*}
    We deduce that for each $i\in S$
    \begin{align*}
        p^{e_i}(\sigma^{p^{a_i}}-1)(\sigma-d)y &=
        p^{e_i}(\sigma^{p^{a_i}}-1)\sum_{j=0}^{m-2} p^jx_j &\mod
        p^{m-1}X\\ &= p^{e_i}(\sigma^{p^{a_i}}-1)\sum_{j\in S}
        p^jx_j&\mod
        p^{m-1}X \\ &= p^{e_i}(\sigma^{p^{a_i}}-1)\sum_{j>i,\ j\in S}
        p^jx_j &\mod p^{m-1}X \\
        &= 0 &\mod p^{m-1}X.
    \end{align*}
    If $a_0=-\infty$, then $0\not\in S$; in this case, define $e_0:=(m-1)-\min\{S\}$, and note that
    \begin{equation*}
        p^{e_0}(\sigma-d)y= p^{e_0}\sum_{j=1}^{m-1} p^jx_j = 0 \mod
        p^{m-1}X.
    \end{equation*}

    Since $g_ww=\proj_W y$ and $W/p^{m-1}W$ is a free
    $R_{m-1}G_{a_{m-1}}$-module generated by $w$, we deduce that for
    all $i\in S$,
    \begin{equation*}
        (\sigma-d)p^{e_i}(\sigma^{p^{a_i}}-1)g_w = 0 \
        \mod (p^{m-1},\sigma^{p^{a_{m-1}}}-1)R_mG.
    \end{equation*}
    Moreover, if $a_0=-\infty$, then
    \begin{equation*}
        (\sigma-d)p^{e_0}g_w= 0 \mod
        (p^{m-1},\sigma^{p^{a_{m-1}}}-1)R_mG.
    \end{equation*}


    Now let $\chi_{a_{m-1}}:R_mG \to R_m G_{a_{m-1}}$ denote the natural map, and we will adopt bar notation to denote the image of an element of $R_mG$ in $R_{m-1}G$.  Then we have shown that if $a_0\neq -\infty$, then
    \begin{equation*}
        \overline{(\sigma-d)\chi_{a_{m-1}}(g_w)} \in \ann_{R_{m-1}G_{a_{m-1}}}
        \langle \{ p^{e_i}(\sigma^{p^{a_i}}-1)\}_{i\in S} \rangle
    \end{equation*}
    while if $a_0=-\infty$,
    \begin{equation*}
        \overline{(\sigma-d)\chi_{a_{m-1}}(g_w)} \in \ann_{R_{m-1}G_{a_{m-1}}}
        \langle p^{e_0}, \{
        p^{e_i}(\sigma^{p^{a_i}}-1)\}_{i\in S} \rangle.
    \end{equation*}

    Case 1: $a_0\neq -\infty$.  Set $t=s-1$ and index the elements
    of $S$ in an increasing manner, thereby defining a one-to-one,
    increasing map $v:\{0,\dots,t\} \to S$. Since $a_0\neq -\infty$,
    $0\in S$ and we have $v(0)=0$.  Set $c_j := a_{v(j)}$ for $j=0,
    \dots, t$, and set
    \begin{equation*}
        b_j := e_{v(j)}= (m-1) - v(j+1)\le m-2
    \end{equation*}
    for $j=0, \dots, t-1$ and $b_t=e(v(t))=0$.  Since by
    (\ref{it:exc.mod.indecom.condition...ai.vs.aj.inequality}), $a_j<a_i$ for $j<i$ elements of $S$,
    $(c_j)\subset \{0,\dots,a_{m-1}-1\}$ is an increasing set and
    $(b_j) \subset \{0,\dots,m-2\}$ is a decreasing set.  Together
    with (\ref{it:exc.mod.indecom.condition...d.in.U1}) and (\ref{it:exc.mod.indecom.condition...power.of.d}), we have satisfied
    conditions (\ref{it:t2}) and (\ref{it:t3}) of
    Lemma~\ref{le:tech}\index{Lemma \ref{le:tech}} for $i=a_{m-1}$.

    Case 2: $a_0= -\infty$. Set $t=s$ and index the elements of $S$
    in an increasing manner, defining a one-to-one, increasing map
    $v:\{1,\dots,t\} \to S$.  Further let $v(0)=0$.  Set $c_j :=
    a_{v(j)}$ for $j=0, \dots, t$, and set
    \begin{equation*}
        b_j := e_{v(j)}= (m-1) -
        v(j+1)\le m-2
    \end{equation*}
    for $j=0, \dots, t-1$ and $b_t=e(v(t))=0$.  Since by
    (\ref{it:exc.mod.indecom.condition...ai.vs.aj.inequality}), $a_j<a_i$ for $j<i$ elements of $S$, $(c_j)
    \subset \{-\infty, 0, \dots,a_{m-1}-1\}$ is an increasing set
    and $(b_j)\subset \{0,\dots,m-2\}$ is a decreasing set.
    Together with (\ref{it:exc.mod.indecom.condition...d.in.U1}) and (\ref{it:exc.mod.indecom.condition...power.of.d}), we have satisfied
    conditions (\ref{it:t2}) and (\ref{it:t3}) of
    Lemma~\ref{le:tech}\index{Lemma \ref{le:tech}} for $i=a_{m-1}$.

    Now in either case, we claim that condition (\ref{it:t4}) for
    $i=a_{m-1}$ holds, as follows.  By (\ref{it:exc.mod.indecom.condition...power.of.d}), there is an
    $R_m$-homomorphism $\phi_d^{(m)}$.  Let $v(l)$, $0\le v(l)<m-1$, be given, and note that if $l=0$, then we can assume $c_0 = a_{v(0)} \neq -\infty$ (since (\ref{it:t4}) imposes no condition in the case $c_0 = -\infty$).  By (\ref{it:exc.mod.indecom.condition...ai.vs.aj.inequality}),
    $$a_{m-1}-c_l=a_{m-1}-a_{v(l)}>m-1-v(l)$$ except if $p=2$, $d\not\in
    U_2$, and $l=a_l=0$.  Hence outside of this exceptional case, we
    may invoke Lemma~\ref{le:phi}\index{Lemma \ref{le:phi}}(\ref{it:phi2}) to see that
    $\phi_d^{(m)}(P(a_{m-1},c_l))\in p^{a_{m-1}-c_l} R_m = p^{a_{m-1}-a_{v(l)}}R_m \subseteq p^{m-v(l)}R_m$.  When $l=0$ then the definition of $v(0)$ gives $\phi_d^{(m)}(P(a_{m-1},c_0)) = 0$. If $l > 0$, then because $v(l) = m-1-b_{l-1}$ we have $\phi_d^{(m)}(P(a_{m-1},c_l)) \in p^{1+b_{l-1}}R_m$.  Therefore condition (\ref{it:t4}) follows outside of the exceptional case $p=2$, $d\not\in
    U_2$, and $l=a_l=0$.

    Now suppose that $p=2$,
    $d\not\in U_2$, and $c_0=0$.  We have from (\ref{it:exc.mod.indecom.condition...ai.vs.aj.inequality}) that
    $a_{m-1}\ge 1$.  If $d=-1$, then again by
    Lemma~\ref{le:phi}\index{Lemma \ref{le:phi}}(\ref{it:phidm1}) we have
    $\phi_d^{(m)}(P(a_{m-1},0))=0$.  Otherwise, $d\in -U_v\setminus
    -U_{v+1}$ for some $v\ge 2$.  Then from
    Lemma~\ref{le:phi}\index{Lemma \ref{le:phi}}(\ref{it:phi3}) we obtain
    $\phi_d(P(a_{m-1},0))=2^{v+a_{m-1}-1} \mod 2^{v+a_{m-1}}$. We need
    that $v+a_{m-1}-1\ge m$.  If $v\ge m$, then since $a_{m-1}\ge 1$ we
    are again done. If $v<m$, then $m-1\ge v$ and from (\ref{it:exc.mod.indecom.condition...ai.vs.v.inequality}) we
    have $a_{m-1}> (m-1)-(v-1) = m-v$.  Hence $a_{m-1}\ge m-v+1$ and
    so $v+a_{m-1}-1\ge m$, as desired.  We have condition
    (\ref{it:t4}) for $i=a_{m-1}$ in all cases.

    In both cases above, then, we may invoke Lemma~\ref{le:tech}\index{Lemma \ref{le:tech}}
    with $i=a_{m-1}$ to assert that in $R_m G_{a_{m-1}}$
    \begin{equation*}
        (\sigma-d)\chi_{a_{m-1}}(g_w)\in \\
        \begin{cases}
            \langle P(a_{m-1},c_0), p^{m-1-b_0}P(a_{m-1},
            c_1),\cdots,& \\ \quad p^{m-1-b_{t-1}}P(a_{m-1},c_t),
            p^{m-1}(\sigma-d)\rangle, & a_0\neq -\infty \\ \langle
            p^{m-1-b_0}P(a_{m-1},c_1),\cdots,& \\ \quad
            p^{m-1-b_{t-1}}P(a_{m-1},c_t), p^{m-1}(\sigma-d)\rangle,
            & a_0 = -\infty.
        \end{cases}
    \end{equation*}

    Now $c_i\le c_t$ for all $0\le i\le t$ and $c_t\ge 0$.  Hence
    in $R_mG_{a_{m-1}}$ we have for all $0\le i\le t$
    \begin{align*}
        (\sigma^{p^{c_t}}-1)P(a_{m-1},c_i) &=
        (\sigma-1)P(c_t,0)P(a_{m-1},c_i) \\
        &= (\sigma-1)P(c_t,c_i)P(c_i,0)P(a_{m-1},c_i) \\
        &= P(c_t,c_i)(\sigma-1)P(a_{m-1},0) \\
        &= 0.
    \end{align*}
    Hence
    \begin{equation*}
        (\sigma^{p^{c_t}}-1)(\sigma-d)\chi_{a_{m-1}}(g_w)\in
        \langle p^{m-1}(\sigma-d)(\sigma^{p^{c_t}}-1)\rangle
        \subset R_mG_{a_{m-1}}.
    \end{equation*}

    Observe that the $c$ of the statement of the Lemma is $c=c_t$.
    Recalling that in $R_mG_{a_{m-1}}/pR_mG_{a_{m-1}}$ we have
    $(\sigma^{p^{c_t}}-1) = (\sigma-1)^{p^{c_t}}$,
    \begin{equation*}
        (\sigma^{p^{c}}-1)(\sigma-d)\chi_{a_{m-1}}(g_w) \in
        \langle p^{m-1}(\sigma-d)(\sigma-1)^{p^{c}}\rangle
        \subset R_mG_{a_{m-1}}.
    \end{equation*}
    Lifting to $R_mG$ gives the desired result.
\end{proof}

We are now prepared to prove the main result of this paper.  

\begin{proof}[Proof of Theorem \ref{th:indecomp}]
    We prove the proposition by induction on $m$, and we start with the case $m=1$.  
    
    Since $\Fp G$ is
    a local ring, cyclic $\Fp G$-modules are indecomposable.  Since
    $d\in U_1$ by (\ref{it:exc.mod.indecom.condition...d.in.U1}), and since the action of $d$ on $X$ is only determined by the congruence class of $d$ modulo $p$ (because we are in the case $m=1$), we see that the relation $(\sigma-d)y=\sum p^i x_i$ reduces to the relation $(\sigma-1)y = x_0$.  But since the additional relation $\sigma^{p^{a_0}}x_0  = x_0$ gives us $(\sigma^{p^{a_0}}-1)x_0 = (\sigma-1)^{p^{a_0}}x_0  = 0$, these two relations together give us 
$$
        X_{\mathbf{a},d,1} :=\langle y,x_0 :
        (\sigma-d)y=x_0,\sigma^{p^{a_0}}x_0=x_0
        \rangle
        =\langle y : (\sigma-1)^{p^{a_0}+1}y=0\rangle.
$$
    Hence $X_{\mathbf{a},d,1}$ is indecomposable and the case
    $m=1$ holds.  Assume then that for some $l\ge 2$, the statement
    of the proposition holds for all $m\le l-1$; we show that the
    statement holds for $m=l$ as well.

    Suppose that $X_{\mathbf{a},d,m}\simeq V\oplus W$ with $V, W \neq
    \{0\}$. In view of the defining relations of $X_{\mathbf{a},d,m}$,
    we have
    \begin{equation*}
        X/p^{m-1}X = A\oplus B,
    \end{equation*}
    where
    \begin{equation*}
        A := X_{\check{\mathbf{a}},d,m-1}
    \end{equation*}
    with $\check{\mathbf{a}} = (a_0,\dots,a_{m-2})$ and
    \begin{equation*}
        B := \langle x_{m-1}\rangle/(p^{m-1}x_{m-1},
        (\sigma^{p^{a_{m-1}}}-1)x_{m-1})\simeq R_{m-1}G_{a_{m-1}}.
    \end{equation*}
    Since conditions (\ref{it:exc.mod.indecom.condition...d.in.U1})--(\ref{it:exc.mod.indecom.condition...ai.vs.v.inequality}) hold for $m$, they
    hold as well for $m-1$ and so by induction, $X_{\check{\mathbf{a}},d,m-1}$
    is an indecomposable $R_{m-1} G$-module.

    Observe that $a_{m-1}=-\infty$ if and only if $B=\{0\}$.  In this
    case, $X/p^{m-1}X$ is indecomposable and we deduce that one of
    $V/p^{m-1}V$ and $W/p^{m-1}W$ is zero, say $V/p^{m-1}V$. But then
    by Nakayama's Lemma, $V=\{0\}$, a contradiction.  Hence if
    $a_{m-1}=-\infty$, then we are done.  We assume then that
    $a_{m-1}\neq -\infty$.  By (\ref{it:exc.mod.indecom.condition...ai.vs.aj.inequality}), $a_0<a_{m-1}$ and so
    $l(y)=p^{a_0}+1\le p^{a_{m-1}}=l(x_{m-1})$.

    Suppose that $p^{a_0}+1=p^{a_{m-1}}$.  Then either $a_0=-\infty$
    and $a_{m-1}=0$, or $p=2$, $a_0=0$, and $a_{m-1}=1$.  In the
    latter case, by (\ref{it:exc.mod.indecom.condition...ai.vs.aj.inequality}) we have a contradiction with
    $a_{m-1}-a_0>m-1$ except in the case $d\not\in U_2$.  Hence
    we are in one of the cases of Lemma~\ref{le:casea}\index{Lemma \ref{le:casea}}, and we
    are done.  We may assume then that $p^{a_0}+1<p^{a_{m-1}}$.

    Since we are in the case $a_{m-1} \neq -\infty$, we know that $B$ is nonzero. As
    a cyclic module over the local ring $R_mG$, it is therefore
    indecomposable. By the Krull-Schmidt Theorem all decompositions of $X/p^{m-1}X$
    into indecomposable $R_{m-1}G$-modules are equivalent. Hence
    $V/p^{m-1}V\oplus W/p^{m-1}W$ decomposes into precisely two
    indecomposable $R_{m-1}G$-modules.  We may assume
    without loss of generality that
    \begin{align*}
        V/p^{m-1}V &\simeq X_{\check a,d,m-1} \\
        W/p^{m-1}W &\simeq B \simeq R_{m-1}G_{a_{m-1}}.
    \end{align*}
    By Lemma~\ref{le:cycprop}\index{Lemma \ref{le:cycprop}}, $W$ is cyclic and we let $w\in W$
    denote a generator.  Since $W/pW\simeq \Fp G_{a_{m-1}}$ and $W$
    is a direct summand of $X$, $l(w)=p^{a_{m-1}}$.  Observe that
    $l(y)=p^{a_0}+1<l(w)$; combined with the previous expression for $l(w)$, we have $l(w) \geq 2$.  Recall also that we are not in the case $p=2$, $a_0=0$ and $a_{m-1}=1$.  


    Our plan is to show that $\langle p^{m-1}x_{m-1}\rangle^\star$ is
    nonzero and lies in each of $V^\star$ and $W^\star$, contradicting
    the independence of $V$ and $W$.

    By Lemma~\ref{le:indecomp2}\index{Lemma \ref{le:indecomp2}}(\ref{it:j4}),
    \begin{equation*}
        \langle p^{m-1}(\sigma-1)^{l(w)-2}w\rangle = \langle
        p^{m-1}(\sigma-1)^{l(w)-2}x_{m-1}\rangle \neq \{0\}.
    \end{equation*}
    Therefore $\langle p^{m-1}(\sigma-1)^{l(w)-2}w\rangle^\star = \langle
        p^{m-1}(\sigma-1)^{l(w)-2}x_{m-1}\rangle^\star$ is a nontrivial subspace of $W^\star.$

    Now $\langle p^{m-1}x_{m-1}\rangle$ is an $\Fp G$-module, and
    the submodules are precisely the images under multiplication by
    powers of the maximal ideal $(\sigma-1)$.  Since $\langle
    p^{m-1}x_{m-1}\rangle^\star$ is a trivial $\Fp G$-module, it
    must be
    \begin{equation*}
        \langle p^{m-1}(\sigma-1)^tx_{m-1}\rangle
    \end{equation*}
     for some $t\in \N$ such that $p^{m-1}(\sigma-1)^{t+1}
    x_{m-1}=0$. Hence the modules $\langle p^{m-1} (\sigma-1)^s
    x_{m-1}\rangle^\star$ are identical for all $0\le s\le t$.  From
    the last paragraph, $t\ge l(w)-2$.

    Putting our results together, we have obtained
    \begin{equation*}
      \{0\} \neq \langle
        p^{m-1}(\sigma-1)^{l(w)-2}x_{m-1}\rangle^\star = \langle p^{m-1}x_{m-1}\rangle^\star \subset
      W^\star.
    \end{equation*}

    Now we approach $V^\star$.  Write $\tilde y=\proj_V y$ and $\check
    y = \proj_W y$.  Then $\check y = g_ww$ for some $g_w\in R_mG$,
    and $\tilde y = y-g_ww$.  Let $\chi_{a_{m-1}}:R_mG \to
    R_mG_{a_{m-1}}$ denote the natural map. Further, we use a bar to denote the image of an
    element in $R_m G_{a_{m-1}}$ under the natural homomorphism $R_m
    G_{a_{m-1}}\to R_{m-1}G_{a_{m-1}}$ modulo $p^{m-1}$.

    We first treat the special case when all $a_i=-\infty$, $0\le i<m-1$,
    and then consider the more generic case.

    Case B1: All $a_i=-\infty$, $0\le i<m-1$.

    Since $(\sigma-d)y =p^{m-1}x_{m-1}$, we obtain by projection that
    \begin{equation*}
        (\sigma-d)g_ww=0 \mod p^{m-1}W.
    \end{equation*}
    Since $W/p^{m-1}W$ is a free $R_{m-1}G_{a_{m-1}}$-module, we
    deduce in $R_{m-1}G_{a_{m-1}}$,
    \begin{equation*}
        \overline{(\sigma-d)\chi_{a_{m-1}}(g_w)} = 0.
    \end{equation*}
    Since $a_{m-1}\neq -\infty$, by (\ref{it:exc.mod.indecom.condition...power.of.d}),
    $d^{p^{a_{m-1}}}\in U_m$.  By Lemma~\ref{le:phidb}\index{Lemma \ref{le:phidb}},
    \begin{equation*}
        \overline{\chi_{a_{m-1}}(g_w)} \in \langle Q_d(a_{m-1},0) \rangle \subset
        R_{m-1}G_{a_{m-1}}.
    \end{equation*}
    Hence
    \begin{equation*}
        \chi_{a_{m-1}}(g_w) = rQ_d(a_{m-1},0)+p^{m-1}s, \quad r,s \in R_mG_{a_{m-1}}.
    \end{equation*}
    Now since $d^{p^{a_{m-1}}}\in U_m$, in $R_{m}G_{a_{m-1}}$
    \begin{equation*}
        (\sigma-d)Q_d(a_{m-1},0) = (\sigma^{p^{a_{m-1}}}-d^{p^{a_{m-1}}})
        =(\sigma^{p^{a_{m-1}}}-1)=0.
    \end{equation*}
    Since $d\in U_1$ by (\ref{it:exc.mod.indecom.condition...d.in.U1}), $p^{m-1}(\sigma-d) =
    p^{m-1}(\sigma-1)$ in $R_mG_{a_{m-1}}$. Therefore
    \begin{equation*}
        (\sigma-d)\chi_{a_{m-1}}(g_w) \in \langle p^{m-1}(\sigma-1)\rangle \subset
        R_{m}G_{a_{m-1}}.
    \end{equation*}
    Hence
    \begin{equation*}
        (\sigma-d)g_w \in \langle
        p^{m-1}(\sigma-1),(\sigma^{p^{a_{m-1}}}-1) \rangle.
    \end{equation*}
    Then by Lemma~\ref{le:indecomp2}\index{Lemma \ref{le:indecomp2}}(\ref{it:j2},\ref{it:j3}), the
    fact that $x_i=0$, $0\le i<m-1$, and $p^{a_{m-1}}\ge 2$, we
    have
    \begin{equation*}
        (\sigma-d)g_ww\in p^{m-1} \langle (\sigma-1)x_{m-1}\rangle.
    \end{equation*}

    Therefore
    \begin{equation*}
        (\sigma-d)\tilde y = (\sigma-d)(y-g_ww)\in
        p^{m-1}x_{m-1} + p^{m-1}\langle
        (\sigma-1)x_{m-1}\rangle.
    \end{equation*}
    Then
    \begin{equation*}
        (\sigma-1)^{l(w)-2}(\sigma-d)\tilde y \in 
        p^{m-1}(\sigma-1)^{l(w)-2}x_{m-1}
        + \langle p^{m-1}(\sigma-1)^{l(w)-1}x_{m-1}
        \rangle.
    \end{equation*}

    Hence, since we have already shown that
    $p^{m-1}(\sigma-1)^{l(w)-2}x_{m-1}\neq 0$, we reach
    \begin{equation*}
        \{0\} \neq \langle p^{m-1}x_{m-1}\rangle^\star =
        \langle p^{m-1}(\sigma-1)^{l(w)-2}x_{m-1}\rangle^\star
        \subset \langle \tilde y \rangle^\star \subset V^\star.
    \end{equation*}

    Case B2: Some $a_i\neq -\infty$, $0\le i<m-1$.

    Let $c=\max\{a_i : 0\le i<m-1\}$. By Lemma~\ref{le:indecomp3}\index{Lemma \ref{le:indecomp3}} and recalling $(\sigma^{p^{c}}-1) = (\sigma-1)^{p^c} \in \F_pG$, we have
    \begin{equation*}
        (\sigma^{p^{c}}-1)(\sigma-d)g_w
        \in \langle (\sigma^{p^c}-1)p^{m-1}(\sigma-d),
        (\sigma^{p^{a_{m-1}}}-1)\rangle.
    \end{equation*}
    We will apply Lemma~\ref{le:indecomp2}\index{Lemma \ref{le:indecomp2}}(\ref{it:j2},\ref{it:j3}).
    Observe that from the defining relations of $X$ and, by
    (\ref{it:exc.mod.indecom.condition...ai.vs.aj.inequality}), the fact that $a_i<a_j$ for all $0\le i<j<m$ with
    $a_i, a_j\neq -\infty$, we have $(\sigma^{p^c}-1)x_i=0$ for
    $0\le i<m-1$. Moreover, $p^{m-1}(\sigma-d)=p^{m-1}(\sigma-1)$ in
    $R_mG$ by (\ref{it:exc.mod.indecom.condition...d.in.U1}), and $p^{m-1}(\sigma^{p^c}-1)=p^{m-1}(\sigma-1)^{p^c}$ in
    $R_mG$. Hence
    \begin{align*}
        (\sigma^{p^{c}}-1)(\sigma-d)g_ww\in &\ \langle
        p^{m-1}(\sigma-1)^{p^c+1} x_{m-1} \rangle + \\ &\
        p^{m-1}\langle
        x_0,\dots,x_{m-2},(\sigma-1)^{p^{a_{m-1}}-1}x_{m-1} \rangle.
    \end{align*}
    Therefore, again using that $(\sigma^{p^c}-1)=(\sigma-1)^{p^c}
    \mod pR_mG$,
    \begin{align}
        (\sigma^{p^{c}}-1)(\sigma-d)\tilde y =\
        &(\sigma^{p^{c}}-1)(\sigma-d)(y-g_ww) \nonumber \\ =\
        &(\sigma^{p^c}-1)\left(\sum_{l=0}^{m-1}p^ix_i\right)-
        (\sigma^{p^c}-1)(\sigma-d)g_ww \nonumber \\ \in\
        &p^{m-1}(\sigma-1)^{p^c}x_{m-1} +  \langle
        p^{m-1}(\sigma-1)^{p^c+1}x_{m-1}\rangle + \nonumber \nonumber\\
        &p^{m-1}\langle
        x_0,\dots,x_{m-2},(\sigma-1)^{p^{a_{m-1}}-1}x_{m-1}\rangle.
        \label{eq:endindec}
    \end{align}

    We intend to apply $(\sigma-1)^{l(w)-2-p^c}$ to \eqref{eq:endindec}.  A proof that $l(w)-2-p^c \geq 0$ will follow from the arguments we give.

    Note that $c=a_i$ for some $i\le m-2$, and $p^c=l(x_i)$.
    We claim that for all $0\le j<m-1$, $l(x_j)\le l(w)-2-p^c$.
	Substituting $l(w)=p^{a_{m-1}}$
    and $c=a_i$, we require that
    \begin{equation*}
        p^{a_j}+p^{a_i}\le p^{a_{m-1}}-2.
    \end{equation*}
    Since we are not in the case $p=2, d\not\in U_2$ and $i=0=a_0$, condition (\ref{it:exc.mod.indecom.condition...ai.vs.aj.inequality}) gives $a_{m-1}-a_i>(m-1)-i\ge 1$. If $p>2$, then
    \begin{equation*}
        p^{a_j}+p^{a_i}\le 2p^{a_i} < p^{a_i+1} <p^{a_i+2}\le
        p^{a_{m-1}}
    \end{equation*}
    and so
    \begin{equation*}
        p^{a_j}+p^{a_i}\le p^{a_{m-1}}-2.
    \end{equation*}
    If, on the other hand, $p=2$, then
    \begin{equation*}
        2^{a_j}+2^{a_i}\le 2\cdot 2^{a_i} = 2^{a_i+1}< 2^{a_i+2}\le
        2^{a_{m-1}}.
    \end{equation*}
    It remains to show only that $2^{a_i+1}+1\neq 2^{a_i+2}$.  But
    considering the parity of each side, we are done. Hence our
    claim that for $0\le j<m-1$, $l(x_j)\le l(w)-2-p^c$, holds.

   We claim that $p^{a_{m-1}}-1>p^c$, for which the previously established inequality $a_{m-1} \geq a_i + 2$ will be useful.  When $p>2$ or $c>0$ we have 
    $p^{a_{m-1}}>p^{a_i+1}=p^{c+1}>p^c+1$; on the other hand, when $p=2$ and $c=0$ one has $2^{a_{m-1}} - 1 \geq 4-1 > 2^0$.  Hence
    $p^{a_{m-1}}-1>p^c$ and then
    \begin{equation*}
        p^{a_{m-1}}-1+l(w)-2-p^c \ge l(w)-1.
    \end{equation*}

    Applying $(\sigma-1)^{l(w)-2-p^c}$ to \eqref{eq:endindec}, then,
    we obtain
    \begin{equation*}
        (\sigma-1)^{l(w)-2-p^c}(\sigma^{p^c}-1)(\sigma-d)\tilde y 
        \in p^{m-1} (\sigma-1)^{l(w)-2}x_{m-1} + \langle
        p^{m-1}(\sigma-1)^{l(w)-1} x_{m-1}\rangle \neq \{0\}.
    \end{equation*}
    Therefore we reach
    \begin{equation*}
        \{0\} \neq \langle p^{m-1}x_{m-1}\rangle^\star = \langle
        p^{m-1}(\sigma-1)^{p^{l(w)-2}}x_{m-1}\rangle^\star \subset
        \langle \tilde y \rangle^\star \subset V^\star.
    \end{equation*}

    In either case, then,
    \begin{equation*}
        V^\star \cap W^\star \neq \{0\},
    \end{equation*}
    a contradiction, and we are done.
\end{proof}

\section{A uniqueness result}\label{sec:uniqueness}
With the indecomposability of $X_{\mathbf{a},d,m}$ established under the hypotheses of Theorem \ref{th:indecomp}, it is natural to ask whether there exist other $\tilde{\mathbf{a}}$ and $\tilde d$ such that $X_{\mathbf{a},d,m} \simeq X_{\tilde{\mathbf{a}},\tilde d,m}$.

\begin{proposition}\label{pr:geta}
Suppose $m, \mathbf{a}$ and $d$ satisfy conditions (\ref{it:exc.mod.indecom.condition...d.in.U1})--(\ref{it:exc.mod.indecom.condition...ai.vs.v.inequality}) of Theorem \ref{th:indecomp}.  Then $\mathbf{a}$ is determined by module theoretic invariants of $X_{\mathbf{a},d,m}$.  In particular, if $\mathbf{a} \neq \tilde{\mathbf{a}}$ and $m,\tilde{\mathbf{a}}$ and $\tilde d$ satisfy conditions (\ref{it:exc.mod.indecom.condition...d.in.U1})--(\ref{it:exc.mod.indecom.condition...ai.vs.v.inequality}), then $X_{\mathbf{a},d,m} \not\simeq X_{\tilde{\mathbf{a}},\tilde d,m}$.
\end{proposition}

\begin{proof}
    The second statement follows easily from first, so we focus on showing that we can recover $\mathbf{a}$ from module-theoretic invariants of $X_{\mathbf{a},d,m}$. We work by induction on $m$.

    First we consider $m=1$.  Since $d\in U_1$ by (\ref{it:exc.mod.indecom.condition...d.in.U1}),
    and taking into account $a_0<n$,
    as an $\Fp G$-module we have
    \begin{equation*}
        X_{\mathbf{a},d,1} :=\left\langle y,x_0 :
        (\sigma-d)y=x_0,\sigma^{p^{a_0}}x_0=x_0 \right\rangle 
        =\left\langle y : (\sigma-1)^{p^{a_0}+1}y=0\right\rangle.
    \end{equation*}
	Then $X$ is a cyclic $\Fp G$-module of dimension $p^{a_0}+1$.
    Now $p^{a_0}+1$, for $a_0\in \{-\infty,0,1,\dots,n-1\}$, uniquely
    determines $a_0$, so we have established our claim for $m=1$.

    Now suppose that for some $m>1$, $X_{\check{\mathbf{a}},d,m-1}$ uniquely
    determines $\check{\mathbf{a}}=(a_0,\cdots,a_{m-2})$.  We show that $X_{\mathbf{a},d,m}$
    uniquely determines $\mathbf{a}$ as well.  Our goal is to
    distinguish $\mathbf{a}$ from abstract properties of $X_{\mathbf{a},d,m}$:
    decompositions into indecomposables; numbers of generators of
    submodules of quotient modules; and dimensions of quotient
    modules.

    In view of the defining relations of $X_{\mathbf{a},d,m}$, we have
    \begin{equation*}
        X_{\mathbf{a},d,m}/p^{m-1}X_{\mathbf{a},d,m} = A\oplus B,
    \end{equation*}
    where
    \begin{equation*}
        A:= X_{\check{\mathbf{a}},d,m-1}
    \end{equation*}
    for $\check{\mathbf{a}} = (a_0,\dots,a_{m-2})$ and
    \begin{equation*}
        B := \langle x_{m-1}\rangle/ (p^{m-1}x_{m-1},
        (\sigma^{p^{a_{m-1}}}-1)x_{m-1}) \simeq R_{m-1}G_{a_{m-1}}.
    \end{equation*}
    (Observe that $a_{m-1}=-\infty$ if and only if $B=\{0\}$.) Since
    conditions  (\ref{it:exc.mod.indecom.condition...d.in.U1})--(\ref{it:exc.mod.indecom.condition...ai.vs.v.inequality}) hold for $m$, they hold
    as well for $m-1$ and so by Theorem \ref{th:indecomp}\index{Theorem \ref{th:indecomp}},
    $X_{\check{\mathbf{a}}, d, m-1}$ is an indecomposable $R_{m-1} G$-module.
    Moreover, $B$ is a free cyclic $R_{m-1}G_{a_{m-1}}$-module
    which, as a cyclic module over the local ring $R_{m-1}G$, is
    indecomposable as well.  By the Krull-Schmidt Theorem, this decomposition into
    indecomposables is unique.

    Observe that by induction $A$ determines $\check{\mathbf{a}} =
    (a_0,\dots,a_{m-2})$ uniquely, and moreover $B$ determines
    $a_{m-1}$ uniquely as the unique element $w$ of
    $\{-\infty,0,1,\dots,n\}$ such that $p^w=\dim_{\Fp} B/pB$.  Our
    strategy is then to decompose $X/p^{m-1}X$ into indecomposables
    and then to determine $\check{\mathbf{a}}$ and $a_{m-1}$ by determining
    which of $V$ and $W$ is $A$ and $B$.

    Therefore let $X/p^{m-1}X=V\oplus W$ be an arbitrary
    decomposition of the $R_{m-1}G$-module into indecomposables.  As
    a first step, we consider the number of generators.  For
    finitely generated $R_mG$-modules $M$, let $g(M)$ denote the
    minimal number of generators of $M$, where we consider the zero
    module $\{0\}$ to have $0$ generators.  Clearly $g(B)=1$ if
    $B\neq \{0\}$ and $g(B)=0$ otherwise.  Now let
    $S=\{a_i\}_{i=1}^{m-2} \setminus \{-\infty\}$.  Considering the
    submodules generated by the elements $y, x_1, \dots, x_{m-2}$,
    we see that $X/pX$ contains $1+\vert S\vert$ direct nontrivial
    summands.  Hence $g(A)\ge 1+\vert S\vert$.

    Now if $g(V)\neq g(W)$, then we have sufficient information to
    identify $V$ and $W$ with the appropriate choices of $A$ and
    $B$, and hence to determine $\mathbf{a}$.  Otherwise $g(V)=g(W)$ and we must have
    $g(V)=g(W)=g(A)=g(B)=1$.  We turn next to the $\Fp$-dimension of
    the modules modulo $p$.  If $g(A)=1$, then $\vert S\vert = 0$ and
    $a_{i}=-\infty$, $1\le i\le m-2$. From the defining relations of
    $A$ we see that as an $\Fp G$-module, $A/pA\simeq \langle y\ :\
    (\sigma-1)^{p^{a_0}+1} y = 0\rangle$. On the other hand,
    $B/pB\simeq \Fp G_{a_{m-1}}$. Then since $a_0<a_{m-1}$ by
    (\ref{it:exc.mod.indecom.condition...ai.vs.aj.inequality}), we have
    \begin{equation*}
        \dim_{\Fp} A/pA=p^{a_0}+1 < p^{a_{m-1}}=\dim_{\Fp} B/pB
    \end{equation*}
    unless either $p=2$, $a_0=0$, and $a_{m-1}=1$, or $a_0=-\infty$, and
    $a_{m-1}=0$ (without condition on $p$).  Hence if we are not in one of these two cases, then again
    we have sufficient information to identify $V$ and $W$, and hence $\mathbf{a}$.

    If, however, we are in the case that $\dim_{\F_p}A/pA = \dim_{\F_p}B/pB$, then either this common quantity is $1$ or $2$.  In the former case we have already seen that $a_{m-1}=0$, and in the latter case we have $a_{m-1}=1$.  Since $\check{\mathbf{a}} = (a_0,\cdots,a_{m-2})$ is determined by induction, this shows that $\mathbf{a}$ is determined in these last cases as well.

\end{proof}

\section{Necessity of conditions}

Theorem \ref{th:indecomp} tells us that if certain conditions on $d$ and $\mathbf{a}$ are met, they are sufficient to determine the indecomposability of $X_{\mathbf{a},d,m}$.  It is natural to ask whether these conditions are also necessary.  Since there are a number of conditions in the theorem, we will only discuss this topic in regards to two of the conditions.

First we consider the necessity of (\ref{it:exc.mod.indecom.condition...a0.bounds}), at least in one particular situation.  Returning to case (\ref{it:casea2}) of Lemma \ref{le:casea}\index{Lemma \ref{le:casea}}, suppose that $p=2$, $d \not\in U_2$, $a_0 = 0$ and $a_{m-1}=1$.  We remarked in the proof of this result that these conditions imply $n>1$ in light of condition (\ref{it:exc.mod.indecom.condition...a0.bounds}).  In fact, if we did not have $n>1$ in this case, then the module would fail to be indecomposable.  To see this, suppose that $p=2$, $n=1$, $d \not\in U_2$, $a_0 = 0$ and $a_{m-1}=1$. Using our calculation of $(\sigma+1)y$ from the  analysis of this case (see equation (\ref{eq:relations.in.case.2})), one then has
$$0= (\sigma^2-1)y= (\sigma-1)(\sigma+1)y= (\sigma-1)(x_0+2^{m-1}x_{m-1}) = 2^{m-1}(\sigma+1)x_{m-1}.$$  One can then prove that
$$\langle y \rangle_{R_mG_1} = \langle y,x_0+2^{m-1}x_{m-1} \rangle_{R_m}$$ is a direct complement to $\langle x_{m-1} \rangle_{R_mG_1}$.

Now we consider the necessity of condition (\ref{it:exc.mod.indecom.condition...ai.vs.aj.inequality}), again in a particular case.  Let $\mathbf{a}=(a_0,a_1,\cdots,a_{m-1})$ and $d$ be given so that there exists some $0 \leq i < m-1$ so that $a_i+((m-1)-i) \geq a_{m-1}$.  Assume furthermore that either $p>2$, or $d \in U_2$, or $i>0$.  We will prove that $X_{\mathbf{a},d,m}$ is decomposable in this case.  Recall that $$X_{\mathbf{a},d,m} = \left\langle y,x_0,x_1,\cdots,x_{m-1}: (\sigma-d)y = \sum_{j=0}^{m-1} p^jx_j; \sigma^{p^{a_j}}x_j = x_j, 0 \leq j < m\right\rangle.$$

By Lemma \ref{le:phi}(\ref{it:phi2}) we know that $\phi_d(P(a_{m-1},a_i)) = p^{a_{m-1}-a_i} \pmod{p^{a_{m-1}-a_i+1}}$.  Since $(m-1)-i \geq a_{m-1}-a_i$, we have $p^{m-1-(a_{m-1}-a_i)} \in \mathbb{Z}$, and so multiplying through in the previous congruence gives $$\phi_d(p^{m-1-(a_{m-1}-a_i)}P(a_{m-1},a_i)) = p^{m-1} \pmod{p^m}.$$  Hence there is some $Q \in R_mG$ so that in $R_mG$ we have $$(\sigma-d)Q = p^{m-1-(a_{m-1}-a_i)}P(a_{m-1},a_i)-p^{m-1}.$$
Now consider the elements $\hat y, \hat x_0,\hat x_1,\cdots, \hat x_{m-1}$ given by
$$\hat y = y - Qx_{m-1}, \quad \hat x_i = x_i + p^{m-1-i-(a_{m-1}-a_i)}P(a_{m-1},a_i)x_{m-1}, \quad \hat x_{m-1} = 0,$$ and with all other $\hat x_j = x_j$.  One can then compute that 
\begin{align*}
(\sigma-d)\hat y &= (\sigma-d)y -(\sigma-d)Qx_{m-1} \\
&= \sum_{j=0}^{m-1} p^j x_j - p^{m-1-(a_{m-1}-a_i)}P(a_{m-1},a_i)x_{m-1}-p^{m-1}x_{m-1} \\
&= \left(\sum_{j\neq i,m-1} p^j x_j\right) + p^i\left(x_i- p^{m-1-i-(a_{m-1}-a_i)}P(a_{m-1},a_i)x_{m-1}\right)+p^{m-1}(x_{m-1}-x_{m-1}) \\
&= \sum_{j=0}^{m-1} p^{j}\hat x_j.
\end{align*}  
Furthermore we have $\sigma^{p^{-\infty}}\hat x_{m-1}=\hat x_{m-1}$ (since $\hat x_{m-1}=0$), and also 
\begin{align*}
(\sigma^{p^{a_i}}-1)\hat x_i &= (\sigma^{p^{a_i}}-1)x_i + (\sigma^{p^{a_i}}-1)p^{m-1-i-(a_{m-1}-a_i)}P(a_{m-1},a_i)x_{m-1} \\&= (\sigma^{p^{a_i}}-1)x_i  + p^{m-1-i-(a_{m-1}-a_i)}(\sigma^{p^{a_{m-1}}}-1)x_{m-1}\\&=0+0.
\end{align*}
This tells us that for the vector $\widehat{\mathbf{a}} = (a_0,a_1,\cdots,a_{m-2},-\infty)$ we get $$X_{\widehat{\mathbf{a}},d,m} \simeq \langle \hat y,\hat x_0,\hat x_1,\cdots,\hat x_{m-1}\rangle.$$
From here it is not hard to show that $$X_{\mathbf{a},d,m} = \langle y,x_0,x_1,\cdots,x_{m-1}\rangle = \langle \hat y,\hat x_0,\cdots,\hat x_{m-1}\rangle \oplus \langle x_{m-1}\rangle \simeq X_{\widehat{\mathbf{a}},d,m} \oplus R_mG_{a_{m-1}}.$$

\section*{Acknowledgements}

We gratefully acknowledge discussions and collaborations with our friends and colleagues S.~Chebolu, F.~Chemotti, I.~Efrat, A.~Eimer, J.~G\"{a}rtner, S.~Gille, P.~Guillot, L.~Heller, D.~Hoffmann, J.~Labute, N.D.~Tan, A.~Topaz, R.~Vakil and K.~Wickelgren which have influenced our work in this and related papers. We would also like to thank the anonymous referee who provided valuable feedback that improved the quality and exposition of this manuscript.


\begin{thebibliography}{MSS}


\bibitem{AnFu} {\sc F.~Anderson}, {\sc K.~Fuller}. \emph{Rings and categories of modules}. Graduate Texts in Mathematics 13. New York: Springer-Verlag, 1973.

\bibitem{BarySorokerJardenNeftin} {\sc L.~Bary-Soroker}, {\sc M.~Jarden}, {\sc D.~Neftin}. The Sylow subgroups of the absolute Galois group $\Gal(\mathbb{Q})$. \emph{Adv.~Math.} {\bf 284} (2015), 186--212.

\bibitem{Benson} {\sc D.~Benson}. \emph{Representations and cohomology I}. Cambridge: Cambridge University Press, 1991.

\bibitem{BLMS.prop.groups} {\sc D.~Benson}, {\sc N.~Lemire}, {\sc J.~Min\'a\v{c}}, {\sc J.~Swallow}. Detecting pro-$p$ groups that are not absolute Galois groups. \emph{J.~Reine Angew.~Math.} {\bf 613} (2007), 175--191.

\bibitem{BS} {\sc J.~Berg}, {\sc A.~Schultz}.  $p$-groups have unbounded realization multiplicity.  \emph{Proc.~AMS}. {\bf 142} (2014), no.~7, 2281--2290.

\bibitem{BLMS} {\sc G.~Bhandari}, {\sc N.~Lemire}, {\sc J.~Min\'a\v{c}}, {\sc J.~Swallow}. Galois module structure of Milnor $K$-theory in characteristic $p$. \textit{New York J.~Math.} \textbf{14} (2008), 215--224.

\bibitem{Bo} {\sc Z.I.~Borevi\v{c}}. The multiplicative group of cyclic $p$-extensions of a local field. (Russian) \textit{Trudy Mat. Inst. Steklov} \textbf{80} (1965), 16--29. English translation, \textit{Proc.~Steklov Inst.~Math.~No.~80 (1965): Algebraic number theory and representations}, edited by D.~K.~Faddeev, 15--30. Providence, RI: American Mathematical Society, 1968.

\bibitem{Bresar} {\sc M.~Bre\v{s}ar}. \emph{Introduction to noncommutative algebra}. Switzerland: Springer International Publishing, 2014.

\bibitem{Carlson} {\sc J.~Carlson}.  \emph{Modules and group algebras}.  Lectures in Mathematics ETH Z\"urich.  Basel: Birkh\"auser Verlag, 1996.

\bibitem{CMSHp3} {\sc S.~Chebolu}, {\sc J.~Min\'{a}\v{c}}, {\sc A.~Schultz}. Galois $p$-groups and Galois modules.  \emph{Rocky Mtn.~J.~Math.}  {\bf 46} (2016), 1405--1446.

\bibitem{CMSS} {\sc F.~Chemotti}, {\sc J.~Min\'a\v{c}}, {\sc A.~Schultz}, {\sc J.~Swallow}. Galois module structure of square power classes for biquadratic extensions.  \emph{manuscript available at \url{https://arxiv.org/abs/2105.13207}.}

\bibitem{DrobotenkoEtAl} {\sc V.~S.~Drobotenko}, {\sc E.~S.~Brobotenko}, {\sc Z.~P.~Zhilinskaya}, {\sc E.~V.~Pogorilyak}. Representations of cyclic groups of prime order $p$ over rings of residue class mod $p^s$. \emph{Amer.~Math.~Soc.~Transl.} {\bf 69}, 241--256.

\bibitem{ErdmannWildon} {\sc K.~Erdmann}, {\sc M.J.~Wildon}. \emph{Introduction to Lie algebras}. London: Springer-Verlag, 2006.

\bibitem{F} {\sc D.K.~Faddeev}. On the structure of the reduced multiplicative group of a cyclic extension of a local field. \textit{Izv.~Akad.~Nauk~SSSR~Ser.~Mat.} \textbf{24} (1960), 145--152.

\bibitem{Gouvea} {\sc F.Q.~Gouv\^{e}a}. \emph{$p$-adic numbers}, 2nd edition. Berlin: Springer (2003).

\bibitem{HannulaThesis} {\sc T.~A.~Hannula}. \emph{Group representations over integers modulo a prime power}. Ph.~D.~thesis, University of Illinois, 1967.

\bibitem{Hannula68} {\sc T.~A.~Hannula}. The integral representation ring $a(R_kG)$. \emph{Trans.~Amer.~Math.~Soc.} {\bf 133} (1968), 553--559.

\bibitem{HMS} {\sc L.~Heller}, {\sc J.~Min\'a\v{c}}, {\sc T.T.~Nguyen}, {\sc A.~Schultz}, {\sc N.D.~T\^an}.  The Galois modules structure of the Artin-Schreier module for finite $p$-extensions. \textit{In preparation.}

\bibitem{Jensen1} {\sc C.U.~Jensen}. On the representations of a group as a Galois group over an arbitrary field. \emph{Th'{e}orie des nombres (Quebec, PQ, 1987)}, 441--458. Berlin: de Gruyter, 1989.

\bibitem{Jensen2} {\sc C.U.~Jensen}. Finite groups as Galois groups over arbitrary fields. \emph{Proceedings of the International Conference on Algebra, Part 2 (Novosibirsk, 1989)}, 435--448. \emph{Contemp. Math.} {\bf 131}, Part 2. Providence, RI: American Mathematical Society, 1992.

\bibitem{Jensen3} {\sc C.U.~Jensen}. Elementary questions in Galois theory. \emph{Advances in algebra and model theory (Essen, 1994; Dresden, 1995)}, 11--24. \emph{Algebra Logic Appl.} {\bf 9}. Amsterdam: Gordon and Breach, 1997.

\bibitem{JensenPrestel1} {\sc C.U.~Jensen}, {\sc A.~Prestel}. Unique realizability of finite abelian 2-groups as Galois groups. \emph{J.~Number Theory} {\bf 40} (1992), no. 1, 12--31.

\bibitem{JensenPrestel2} {\sc C.U.~Jensen}, {\sc A.~Prestel}. How often can a finite group be realized as a Galois group over a field?
\emph{Manuscripta Math.} {\bf 99} (1999), 223--247.

\bibitem{JensenLedetYui} {\sc C.U.~Jensen}, {A.~Ledet}, {N.~Yui}. \emph{Generic polynomials: constructive aspects of the inverse Galois problem}. Mathematical Sciences Research Institute Publications 45. Cambridge: Cambridge University Press, 2002.

\bibitem{Katok} {\sc S.~Katok}. \emph{$p$-adic analysis compared with real}. Providence, RI: AMS, 2007.

\bibitem{Kiming} {\sc I.~Kiming}. Explicit classifications of some $2$-extensions of a field of characteristic different from $2$. \emph{Canad.~J.~Math.} {\bf 42} (1990), no.~5, 825--855.

\bibitem{Koblitz} {\sc N.~Koblitz}. \emph{$p$-adic numbers, $p$-adic analysis, and zeta-functions}, 2nd edition. New York: Springer-Verlag, 1984.

\bibitem{Labute} {\sc J.P.~Labute}. Demu\v{s}kin groups of type $\aleph_0$. \emph{Bull.~Soc.~Math.~France} {\bf 94} (1966), 211--244.

\bibitem{Ledet} {\sc A.~Ledet}. On $2$-groups as Galois groups. \emph{Canad.~J.~Math.} {\bf 47} (1995), no.~6, 1253--1273.

\bibitem{LMSS}  {\sc N.~Lemire}, {\sc J.~Min\'a\v{c}}, {\sc A.~Schultz}, {\sc J.~Swallow}. Galois module structure of Galois cohomology for embeddable cyclic extensions of degree $p^n$. \emph{J.~London Math.~Soc.} {\bf 81} (2010), no.~3, 525--543.

\bibitem{LMS} {\sc N.~Lemire}, {\sc J.~Min\'a\v{c}}, {\sc J.~Swallow}. Galois module structure of Galois cohomology and partial Euler-Poincar\'{e} characteristics. \textit{J. Reine Angew. Math.} {\bf 613} (2007), 147--173.


\bibitem{Lorenz} {\sc F.~Lorenz}. Ein Scholion zum Satz 90 von Hilbert. \emph{Anh.~Math.~Sem.~Univ.~Hamburg} {\bf 68} (1998), 347--362.

\bibitem{MSS.auto} {\sc J.~Min\'a\v{c}}, {\sc A.~Schultz}, {\sc J.~Swallow}. Automatic realizations of Galois groups with cyclic quotient of order $p^n$.  \emph{J.~Th\'{e}or.~Nombres Bordeaux} {\bf 20} (2008), 419--430.

\bibitem{MSS.K.in.char.p} {\sc J.~Min\'a\v{c}}, {\sc A.~Schultz}, {\sc J.~Swallow}. Galois module structure of Milnor $K$-theory mod $p^s$ in characteristic $p$.  \emph{New York J.~Math.} {\bf 14} (2008), 225--233.

\bibitem{MSS} {\sc J.~Min\'a\v{c}}, {\sc A.~Schultz}, {\sc J.~Swallow}. Galois module structure of $p$th-power classes of cyclic extensions of degree $p^n$.  \emph{Proc.~London Math.~Soc.} \textbf{96} (2006), 307--341.

\bibitem{MSS2b} {\sc J.~Min\'a\v{c}}, {\sc A.~Schultz}, {\sc J.~Swallow}. Galois module structure of the units modulo $p^m$ of cyclic extensions of degree $p^n$. \emph{manuscript available at \url{https://arxiv.org/abs/2105.13216}}.

\bibitem{MSS2c} {\sc J.~Min\'a\v{c}}, {\sc A.~Schultz}, {\sc J.~Swallow}. Arithmetic properties encoded in the Galois module structure of $K^\times/K^{\times p^m}$. \emph{manuscript available at \url{https://arxiv.org/abs/2105.13221}}.

\bibitem{MS} {\sc J.~Min\'a\v{c}}, {\sc J.~Swallow}.  Galois module structure of $p$th-power classes of extensions of degree $p$. \textit{Israel J.~Math.} {\bf 138} (2003), 29--42.

\bibitem{MS2} {\sc J.~Min\'a\v{c}}, {\sc J.~Swallow}.  Galois embedding problems with cyclic quotient of order $p$. \textit{Israel J.~Math.}, {\bf 145} (2005), 93--112.

\bibitem{Nakano} {\sc D.K.~Nakano}. Cohomology of algebraic groups, finite groups, and Lie algebras: Interactions and Connections. In \emph{Lie Theory and Representation Theory}, Survey of Modern Mathematics Vol.~II (edited by N. Hu, B. Shu, J.P. Wang), International Press, 2012, pp. 151--176.

\bibitem{Schultz} {\sc A.~Schultz}. Parameterizing solutions to any Galois embedding problem over $\mathbb{Z}/p^n\mathbb{Z}$ with elementary $p$-abelian kernel. \textit{J.~Alg.} {\bf 411} (2014), 50--91.


\bibitem{Szekeres} {\sc G.~Szekeres}. Determination of a certain family of finite metabelian groups. \emph{Trans.~Amer.~Math.~Soc.} {\bf 66} (1949), 1--43.

\bibitem{Thevenaz81} {\sc J.~Th\'{e}venaz}. Representations of finite groups in characteristic $p^r$.  \emph{J.~Alg.} {\bf 72} (1981), 478--500.

\bibitem{Waterhouse} {\sc W.~Waterhouse}. The normal closures of certain Kummer extensions. \textit{Canad.~Math.~Bull.} {\bf 37} (1994), no.~1,133--139.


\end{thebibliography}
\end{document}